\RequirePackage{fix-cm}
\documentclass[smallextended]{svjour3}       % onecolumn (second format)
\smartqed  % flush right qed marks, e.g. at end of proof
\usepackage{graphicx}
\usepackage{amsfonts}
\usepackage{amsmath}
\usepackage{amssymb}
\usepackage{amscd,multirow}
\usepackage{graphicx}
\usepackage{epstopdf}
\usepackage{subfigure}
\usepackage{subeqnarray,cases}
\usepackage{hyperref}
\usepackage[misc]{ifsym}
\usepackage[justification=centering]{caption}

\usepackage[ruled]{algorithm2e}
\begin{document}

\title{Inertial accelerated primal-dual methods
for linear equality constrained convex optimization problems \thanks{This work was supported by  the National Natural Science  Foundation of China (11471230).}
}


\author{Xin He  \and
        Rong Hu 	 \and
        Ya-Ping Fang
}


\institute{Xin He \at
              Department of Mathematics, Sichuan University, Chengdu, Sichuan, P.R. China \\
              \email{hexinuser@163.com}           %  \\
           \and
           \ Rong Hu \at
              Department of Applied Mathematics, Chengdu University of Information Technology, Chengdu, Sichuan, P.R. China\\
           \email{ronghumath@aliyun.com}
           \and
           Ya-Ping Fang \Letter  \at
           Department of Mathematics, Sichuan University, Chengdu, Sichuan, P.R. China \\
              \email{ypfang@scu.edu.cn}
}

\date{Received: date / Accepted: date}

\maketitle

\begin{abstract}
In this paper, we  propose an inertial  accelerated primal-dual method  for the linear equality constrained convex optimization problem.  When the objective function has a ``nonsmooth + smooth'' composite structure,  we further propose an inexact inertial primal-dual method by linearizing the smooth individual  function  and  solving the subproblem inexactly.  Assuming merely convexity, we prove that the proposed methods  enjoy  $\mathcal{O}(1/k^2)$ convergence rate on the objective residual and the feasibility violation in the primal model.  Numerical results are reported to demonstrate the validity of the  proposed  methods.

\end{abstract}

 \section{Introduction}

Consider the linear equality constrained convex optimization problem:
	\begin{equation}\label{ques_1}
				\min_{x}  \quad F(x), \quad s.t.  \  Ax = b,
	\end{equation}
where $F: \mathbb{R}^{n}\to\mathbb{R}$ is a  closed convex but possibly nonsmooth function, $A\in\mathbb{R}^{m\times n}$ and $b\in\mathbb{R}^{m}$.  The problem \eqref{ques_1} captures a number of important applications arising in various areas, and the following are three concrete examples.

\noindent {\bf Example 1.1} The basis pursuit problem (see e.g.\cite{Candes2008,Chen2001}):
\begin{equation}\label{ques_2}
		\min_x  \quad \|x\|_1, \quad s.t.  \  Ax = b,
\end{equation}
where  $A\in\mathbb{R}^{m\times n}$ with $m \ll n$, and $\|\cdot\|_1$ is the $\ell_1$-norm of $\mathbb{R}^{n}$ defined by $\|x\|_1 =\sum^n_{i=1} |x_i|$. Algorithms for the basis pursuit problem can  be  found in \cite{Van2009} and \cite{YinO2008}.

\noindent {\bf Example 1.2} The linearly constrained $\ell_1-\ell_2$ minimization problem \cite{KangYW2013}:
\begin{equation}\label{ques_3}
		\min_x  \quad \|x\|_1+\frac{\beta}{2}\|x\|^2_2, \quad s.t.  \  Ax = b,
\end{equation}
where $\beta>0$ and $\|\cdot\|_2$ is the $\ell_2$-norm of $\mathbb{R}^{n}$ defined by $\|x\|^2_2 =\sum^n_{i=1} x_i^2$. When $\beta$ is small enough, a solution of the problem \eqref{ques_3}  is also a solution of the basis pursuit problem \eqref{ques_2}. Since the problem \eqref{ques_3} has the regularization term $\frac{\beta}{2}\|x\|^2_2$,  it is less sensitive to noise than the basis pursuit problem \eqref{ques_2}.

%, then by solving problem \eqref{ques_3}, large sparse signals can be recovered from highly incomplete information (see \cite{}).

\noindent {\bf Example 1.3} The global consensus problem \cite{BoydP2010}:
\begin{equation*}\label{ques_5}
		\min_{X\in\mathbb{R}^{n\times n}}  \quad F(X) = \sum^N_{i=1}f_i(X_i), \quad s.t.  \ X_i = X_j,\quad \forall i, j\in\{1,2,\cdots N\},
\end{equation*}
where $f_i:\mathbb{R}^n\to\mathbb{R}$ is convex, $ i=1,2,\cdots, N$. The global consensus problem is a widely investigated model that has important applications in signal processing \cite{NedicO2010}, routing of wireless sensor networks \cite{Madan2006} and optimal consensus of agents \cite{ShiJ2012}.

Recall that  $(x^*, \lambda^*)\in \mathbb{R}^n \times \mathbb{R}^m$ is a KKT point of the problem \eqref{ques_1}  if
\begin{equation}\label{saddle_point}
	\begin{cases}
		-A^T\lambda^* \in  \partial F(x^*),\\
		Ax^* -b =0,
	\end{cases}
\end{equation}
where   $\partial F$ is the classical  subdifferential of $F$ defined by
 \[\partial F(x) = \{v\in\mathbb{R}^n | F(y)\geq F(x)+\langle v,y-x \rangle,\quad \forall y\in\mathbb{R}^n\}.\]
Let $\Omega$ be the KKT point set of the problem \eqref{ques_1}. It is well-known that  $x^*$ is a solution of the problem \eqref{ques_1}  if and only if there exists $\lambda^*\in\mathbb{R}^{m}$ such that $(x^*, \lambda^*)\in\Omega$ if and only if
\begin{equation*}\label{eq:intr_1}
	\mathcal{L}(x^*,\lambda)\leq  \mathcal{L}(x^*,\lambda^*)\leq  \mathcal{L}(x,\lambda^*), \qquad \forall (x,\lambda)\in \mathbb{R}^{n}\times\mathbb{R}^m,
\end{equation*}
where $\mathcal{L}:\mathbb{R}^{n}\times\mathbb{R}^m\to\mathbb{R}$ is the Lagrangian function associated with the problem \eqref{ques_1} defined by
\[ \mathcal{L}(x,\lambda) = F(x)+\langle \lambda,Ax-b\rangle.\]
A classical method for solving the problem \eqref{ques_1} is the augmented Lagrangian method (ALM) \cite{Bertsekas1982}:
\begin{eqnarray}\label{al:in1}
	\begin{cases}
		x_{k+1}\in\mathop{\arg\min}_x \mathcal{L}(x,\lambda_k)+\frac{\sigma}{2}\|Ax-b\|^2,\\
		\lambda_{k+1} = \lambda_k+\sigma(Ax_{k+1}-b),
	\end{cases}
\end{eqnarray}
In general, since $\mathcal{L}(x,\lambda_k)+\frac{\sigma}{2}\|Ax-b\|^2$ is not strictly convex, the subproblem  may  have more than one solutions and be difficult to solve. To overcome this disadvantage, the proximal ALM \cite{ChenT1994} has been proposed:
\begin{eqnarray}\label{al:prox}
	\begin{cases}
		x_{k+1}=\mathop{\arg\min}_x F(x)+\langle A^T\lambda_k,x\rangle +\frac{\sigma}{2}\|Ax-b\|^2+\frac{1}{2}\|x-x_k\|_P^2,\\
		\lambda_{k+1} = \lambda_k+\sigma(Ax_{k+1}-b),
	\end{cases}
\end{eqnarray}
where $\|x\|^2_P = x^TPx$ with a positive semidefinite matrix $P$ and $P+\sigma A^TA$ is positive definite.

In some practical situations, the objective function $F$ has the composite structure: $F(x) = f(x) + g(x)$, where $f$ is a  convex but possibly nonsmooth function and $g$ is a convex   smooth function. Then the problem \eqref{ques_1} becomes  the linearly constrained composite convex optimization problem:
\begin{equation}\label{ques_4}
				\min_{x}  \quad f(x)+g(x), \quad s.t.  \  Ax = b.
\end{equation}
An application of the  method \eqref{al:prox}  to the problem \eqref{ques_4} with linearizing the smooth function $g$ leads to  the linearized ALM \cite{Xu2017}:
\begin{eqnarray}\label{al:in2}
	\qquad\begin{cases}
		x_{k+1}\in\mathop{\arg\min}_x f(x)+\langle \nabla g(x_k)+A^T\lambda_k,x\rangle +\frac{\sigma}{2}\|Ax-b\|^2+\frac{1}{2}\|x-x_k\|_P^2,\\
		\lambda_{k+1} = \lambda_k+\sigma(Ax_{k+1}-b).
	\end{cases}
\end{eqnarray}

\subsection{Related works}
 Under the assumption that $F$ is smooth, He and Yuan \cite{HeYuan2010} showed that the iteration-complexity of the method \eqref{al:in1} is $\mathcal{O}(1/k)$ in terms of the objective residual of the associated $\mathcal{L}(x,\lambda)$. When $F$ is nonsmooth, Gu et al. \cite{Gu2014} proved  that the method \eqref{al:in1} enjoys a worst-case $\mathcal{O}(1/k)$ convergence rate in the ergodic sense.  A worst-case $\mathcal{O}(1/k)$ convergence rate in the non-ergodic sense of the method \eqref{al:prox} was shown in \cite{Ma2018}.  When $g$ has a Lipschitz continuous gradient with constant $L_g$ and $P\succ L_g Id$,  Xu \cite{Xu2017} proved that  the method \eqref{al:in2} achieves $\mathcal{O}(1/k)$  convergence rate in the ergodic sense. Tran-Dinh and Zhu \cite{Tran2020} proposed a modified version of the method \eqref{al:in2}  and proved that  the objective residual and feasibility violation sequences generated by the method both enjoy $\mathcal{O}(1/k)$  non-ergodic convergence rate. Liu et al. \cite{Liu2019} investigated the nonergodic convergence rate of an inexact augmented Lagrangian method for the problem \eqref{ques_4}.

%  For more literature on unaccelerated version of the algorithms \eqref{al:in1}, \eqref{al:prox} and  \eqref{al:in2} to solve problems \eqref{ques_1} and \eqref{ques_4}, we refer \cite{Hajinezhad2019,HeYZ2013}. and other references therein.

Generally, naive first-order methods converge slowly. Much effort has been made to accelerate the  existing first-order methods in past decades. Nesterov \cite{Nesterov1983}  first  proposed an accelerated version of the classical gradient method for a  smooth convex optimization problem, and proved that the accelerated inertial gradient method enjoys $\mathcal{O}(1/k^2)$ convergence rate.  Beck and Teboulle \cite{BeckT2009} proposed an iterative shrinkage-thresholding algorithm  for solving the linear inverse problem, which achieves $\mathcal{O}(1/k^2)$ convergence rate. The acceleration idea of \cite{Nesterov1983} was further  applied  in Nesterov \cite{Nesterov20132}  to design the accelerated methods for  unconstrained convex  composite optimization problems.  Su et al. \cite{SuBC2016} first studied accelerated methods from a continuous-time perspective.   Since then,  some new accelerated inertial methods based on the second-order dynamical system  have been  proposed for unconstrained optimization problems (see e.g. \cite{Apidopoulos2020,AttouchP2016,AttouchPJ2014}).  For more results on inertial methods for unconstrained optimization problems, we refer the reader to \cite{AttouchCPR2018,Goldfarb2013,NocedalJ2006}.
	
Meanwhile, inertial  accelerated methods  for  linearly constrained optimization problems have also been well-developed.  He and Yuan \cite{HeYuan2010} proposed an accelerated inertial ALM for the problem \eqref{ques_1} and proved  that its convergence rate is $\mathcal{O}(1/k^2)$  by using an extrapolation  technique similar to \cite{BeckT2009}.  Kang et al. \cite{KangJ2015} presented an inexact version of the accelerated ALM with  inexact  calculations of subproblems and  showed that the convergence rate   remains $\mathcal{O}(1/k^2)$ under the assumption that $F$ is strongly convex.
Kang et al. \cite{KangYW2013} further presented an accelerated Bregman method for  the linearly constrained $\ell_1-\ell_2$ minimization problem, and a convergence rate of $\mathcal{O}(1/k^2)$ was proved when the accelerated Bregman method is applied to solve the problem \eqref{ques_1}. To linearize the augmented term of the Bregman method, Huang et al. \cite{Huang2013} raised an accelerated linearized Bregman algorithm  with  $\mathcal{O}(1/k^2)$  convergence rate. For the problem \eqref{ques_4}, Tran-Dinh and Zhu \cite{Tran2020} proposed an inertial primal-dual method  which enjoys $\underline{o}(1/k\sqrt{\log k})$ convergence rate.  Xu \cite{Xu2017} proposed an accelerated version of the linearized  ALM \eqref{al:in2}, named the accelerated linearized augmented Lagrangian method,  which is formulated as follows:
\begin{eqnarray}\label{al0}
	\qquad\begin{cases}
		\hat{x}_k=(1-\alpha_k)\bar{x}_k+\alpha_k x_k,\\
		x_{k+1} \in  \mathop{\arg\min}_x f(x)+ \frac{\beta_k}{2}\|Ax-b\|^2+\frac{1}{2}\|x-x_k\|_{P_k}^2+\langle \nabla g(\hat{x}_k)+A^T{\lambda}_k, x\rangle,\\
		\bar{x}_{k+1}=(1-\alpha_k)\bar{x}_k+\alpha_k x_{k+1},\\
      \lambda_{k+1} =\lambda_k+ \gamma_k(Ax_{k+1}-b).
	\end{cases}
\end{eqnarray}
 It was shown in Xu \cite{Xu2017} that the algorithm \eqref{al0} enjoys $\mathcal{O}(1/k^2)$  convergence rate under specific parameter settings. It is worth mentioning that  to achieve the  $\mathcal{O}(1/k^2)$ rate,   linearization to the augmented term is not allowed in the algorithm \eqref{al0}  since it may cause great difficulty on solving subproblems.  Xu \cite{Xu2017} did not discuss the convergence analysis of the method  when the subproblem is solved inexactly.

\subsection{Inertial primal-dual methods}
We first propose Algorithm \ref{al:al1}, an inertial version of  the proximal ALM \eqref{al:prox},  for solving the problem \eqref{ques_1}.  Algorithm \ref{al:al1} is  inspired by  the  second-order primal-dual  dynamical system in \cite{HeSiam,Zeng2019} and the Nesterov accelerated methods for unconstrained optimization problem \cite{AttouchCPR2018,BeckT2009,Nesterov1983}.  When the objective has the composite structure: $F(x)=f(x)+g(x)$,  by linearizing the smooth function $g$ and  introducing the  perturbed sequence $\{\epsilon_k\}_{k\geq 1}$  in Step 2 of Algorithm \ref{al:al1},   we propose an inexact inertial proximal primal-dual method (Algorithm \ref{al:al2}) for  the problem \eqref{ques_4}. As a comparison to Algorithm \ref{al:al1},  we  solve the subproblem inexactly by finding an approximate solution instead of an exact solution.

\begin{algorithm}
%\SetAlgoNoLine
        \caption{Inertial proximal primal-dual method for the problem \eqref{ques_1}}
        \label{al:al1}
%        \KwIn{this text}
%        \KwOut{how to write algorithm with \LaTeX2e }
        {\bf Initialization:} Choose $x_0\in\mathbb{R}^n,\ \lambda_0\in\mathbb{R}^m$. Set $x_1 =x_0,\ \lambda_1=\lambda_0$, $M_0\in\mathbb{S}_+(n)$. Choose parameters $s>0, \alpha\geq 3$. \\
        \For{$k = 1, 2,\cdots$}{
       {\bf Step 1:} Compute $\bar{x}_k = x_k+\frac{k-2}{k+\alpha-2}(x_k-x_{k-1}),\quad \bar{\lambda}_k =  \lambda_k+\frac{k-2}{k+\alpha-2}( \lambda_k- \lambda_{k-1}).$\\
       {\bf Step 2:} Set
        \[ \hat{\lambda}_k =\frac{k+\alpha-2}{\alpha-1}\bar{\lambda}_{k}-\frac{k-1}{\alpha-1}\lambda_k, \quad\eta_k = \frac{k-1}{k+\alpha-2}Ax_{k}+\frac{\alpha-1}{k+\alpha-2}b.\]
     \\
  \quad Choose  $M_k\in\mathbb{S}_+(n)$ and update
        \[x_{k+1} \in \mathop{\arg\min}_x F(x)+\frac{k+\alpha-2}{2sk}\|x-\bar{x}_k\|_{M_k}^2+\frac{sk(k+\alpha-2)}{2(\alpha-1)^2}\|Ax-\eta_k\|^2+\langle  A^T\hat{\lambda}_k, x\rangle.\]
  {\bf Step 3:} $\lambda_{k+1} =\bar{\lambda}_k+ \frac{sk}{k+\alpha-2}(Ax_{k+1}-b+\frac{k-1}{\alpha-1}A(x_{k+1}-x_k))$,\\
\If{A stopping condition is satisfied}{Return $(x_{k+1},\lambda_{k+1})$}
}
\end{algorithm}
\begin{algorithm}
        \caption{Inexact inertial linearized  proximal primal-dual method for the problem  \eqref{ques_4}}
        \label{al:al2}
        {\bf Initialization:} Choose $x_0\in\mathbb{R}^n,\ \lambda_0\in\mathbb{R}^m$. Set $x_1 =x_0,\ \lambda_1=\lambda_0$, $M_0\in\mathbb{S}_+(n)$, $\epsilon_0=0$. Choose parameters $s>0$, $\alpha\geq 3$. \\
        \For{$k = 1, 2,\cdots$}{
       {\bf Step 1:} Compute $\bar{x}_k = x_k+\frac{k-2}{k+\alpha-2}(x_k-x_{k-1}),\quad \bar{\lambda}_k =  \lambda_k+\frac{k-2}{k+\alpha-2}( \lambda_k- \lambda_{k-1}).$\\
       {\bf Step 2:} Set
        \[ \hat{\lambda}_k =\frac{k+\alpha-2}{\alpha-1}\bar{\lambda}_{k}-\frac{k-1}{\alpha-1}\lambda_k, \quad\eta_k = \frac{k-1}{k+\alpha-2}Ax_{k}+\frac{\alpha-1}{k+\alpha-2}b.\]
     \\
   \qquad Choose  $M_k\in\mathbb{S}_+(n)$, $\epsilon_k\in\mathbb{R}^n$ and update
        \[x_{k+1} \in \mathop{\arg\min}_x f(x)+\frac{k+\alpha-2}{2sk}\|x-\bar{x}_k\|_{M_k}^2+\frac{sk(k+\alpha-2)}{2(\alpha-1)^2}\|Ax-\eta_k\|^2+\langle \nabla g(\bar{x}_k)+ A^T\hat{\lambda}_k-\epsilon_k , x\rangle.\]
  {\bf Step 3:} $\lambda_{k+1} =\bar{\lambda}_k+ \frac{sk}{k+\alpha-2}(Ax_{k+1}-b+\frac{k-1}{\alpha-1}A(x_{k+1}-x_k))$,\\
\If{A stopping condition is satisfied}{Return $(x_{k+1},\lambda_{k+1})$}
}
\end{algorithm}

 \subsection{Outline}
 The rest of the paper is organized as follows. In Section 2, we investigate the convergence analysis of the proposed methods. In Section 3, we performed  numerical experiments.  Finally, we give a concluding remark in Section 4.

\section{Convergence analysis}

In this section we analyze the convergence rates of Algorithm \ref{al:al1} and  Algorithm \ref{al:al2}. Assuming merely convexity, we show that both of them  enjoy $\mathcal{O}(1/k^2)$ convergence rates in terms of the objective function and the primal feasibility.

To do so, we first recall some standard notations and results which will be used in the paper.  In what follows, we always use $\|\cdot\|$ to denote the $\ell_2$-norm.
% A  function $F:\mathbb{R}^n\to\mathbb{R}$ is coercive if
%	\[\lim_{\|x\|\to +\infty}F(x)=+\infty.\]
Let $\mathbb{S}_+(n)$ denote the set of all positive semidefinite matrixes in $\mathbb{R}^{n\times n}$ and $Id$ is the identity matrix. For $M\in \mathbb{S}_+(n)$, we introduce the semi-norm on $\mathbb{R}^n$: $\|x\|_M =\sqrt{x^TMx}$ for any $x\in\mathbb{R}^n$. This introduces on $\mathbb{S}_+(n)$ the following partial ordering: for any $M_1, M_2 \in \mathbb{S}_+(n)$,
 \[M_1 \succcurlyeq M_2 \Longleftrightarrow \|x\|_{M_1}\geq \|x\|_{M_2},\quad \forall x\in \mathbb{R}^{n}.\]
For any $x,y\in\mathbb{R}^n$, the following equality holds:
\begin{eqnarray}
	\frac{1}{2}\|x\|_M^2-\frac{1}{2}\|y\|_M^2 = \langle x, M(x-y)\rangle-\frac{1}{2}\|x-y\|_M^2,  \quad \forall M\in\mathbb{S}_+(n).\label{eq:known_2}
\end{eqnarray}

Now, we start to analyze  Algorithm \ref{al:al1}.
\begin{lemma}\label{pro:al1}
	Let $\{(x_k,\lambda_k,\bar {x}_k)\}_{k\geq 1}$ be the sequence generated by Algorithm \ref{al:al1}. Then
\begin{eqnarray}\label{sq:sequ1}
 \frac{k+\alpha-2}{k}M_k(x_{k+1}-\bar{x}_k) \in  - s\left(\partial F(x_{k+1})+A^T(\lambda_{k+1}+\frac{k-1}{\alpha-1}(\lambda_{k+1}-\lambda_{k}))\right).
\end{eqnarray}
\end{lemma}
 \begin{proof}
 From step 2, we have
 	 \[ 0\in \partial F(x_{k+1})+\frac{k+\alpha-2}{sk}M_k(x_{k+1}-\bar{x}_k)+\frac{sk(k+\alpha-2)}{(\alpha-1)^2} A^T(Ax_{k+1}-\eta_k)+A^T\hat{\lambda}_k.\]
 This yields
 \begin{equation}\label{eq:pro1_1}
 	\frac{k+\alpha-2}{k}M_k(x_{k+1}-\bar{x}_k)\in -s\left(\partial F(x_{k+1})+A^T\left(\frac{sk(k+\alpha-2)}{(\alpha-1)^2} (Ax_{k+1}-\eta_k)+\hat{\lambda}_k\right)\right).
 \end{equation}
It follows from Step 2 and Step 3 that
 \begin{eqnarray*}\label{eq:pro1_2}
 	 &&	 \frac{sk(k+\alpha-2)}{(\alpha-1)^2}(Ax_{k+1}-\eta_k)+\hat{\lambda}_k\nonumber\\
 	 && \qquad= \frac{sk(k+\alpha-2)}{(\alpha-1)^2} A x_{k+1}-\frac{sk(k-1)}{(\alpha-1)^2}Ax_{k}-\frac{sk}{\alpha-1}b+\hat{\lambda}_k\nonumber  \\
 	 &&\qquad = \frac{sk}{\alpha-1}(Ax_{k+1}-b+\frac{k-1}{\alpha-1}A(x_{k+1}-x_k))+\hat{\lambda}_k\nonumber\\
 &&\qquad = \frac{k+\alpha-2}{\alpha-1}(\lambda_{k+1}-\bar{\lambda}_k)+ \frac{k+\alpha-2}{\alpha-1}\bar{\lambda}_{k}-\frac{k-1}{\alpha-1}\lambda_k \\
 &&\qquad = \lambda_{k+1} +\frac{k-1}{\alpha-1}(\lambda_{k+1}-\lambda_k). \nonumber
 \end{eqnarray*}
This together with \eqref{eq:pro1_1} yields  \eqref{sq:sequ1}.
\end{proof}

\begin{lemma}\label{le_new2}
Suppose that $F$ is a  closed convex function, $\Omega\neq\emptyset$ and $M_{k-1}\succcurlyeq M_k$. Let $\{(x_k,\lambda_k, \bar{x}_k, \hat{\lambda}_k)\}_{k\geq 1}$ be the sequence generated by Algorithm \ref{al:al1} and $(x^*,\lambda^*)\in\Omega$.  Define
\begin{eqnarray}\label{energy1}
	\mathcal{E}_k = \frac{s(k^2-k)}{(\alpha-1)^2}(\mathcal{L}(x_k,\lambda^*)-\mathcal{L}(x^*,\lambda^*))+ \frac{1}{2}\|\hat{x}_k-x^*\|_{M_{k-1}}^2+\frac{1}{2}\|\hat{\lambda}_k-\lambda^*\|^2
\end{eqnarray}
with
\begin{equation}\label{hat_x}
	\hat{x}_k = \frac{k+\alpha-2}{\alpha-1}\bar{x}_k-\frac{k-1}{\alpha-1}x_k.
\end{equation}
Then, for any $k\geq 1$, we have
\begin{eqnarray*}
	\mathcal{E}_{k+1} &\leq& \mathcal{E}_k-\frac{(k+\alpha-1)^2}{2(\alpha-1)^2}(\|x_{k+1}-\bar{x}_k\|_{M_k}^2+\|\lambda_{k+1}-\bar{\lambda}_k\|^2).
\end{eqnarray*}
\end{lemma}
\begin{proof}
By computation,
 \begin{eqnarray}\label{eq:th1_3}	
	\hat{x}_{k+1} &=& \frac{k+\alpha-1}{\alpha-1}(x_{k+1}+\frac{k-1}{k+\alpha-1}(x_{k+1}-x_k))-\frac{k}{\alpha-1}x_{k+1}\nonumber \\
	 &=& \frac{k+\alpha-2}{\alpha-1}x_{k+1}-\frac{k-1}{\alpha-1}x_k\\
	 &=& \frac{k+\alpha-2}{\alpha-1}(x_{k+1}-\bar{x}_k)+ \frac{k+\alpha-2}{\alpha-1}\bar{x}_k-\frac{k-1}{\alpha-1}x_k\nonumber\\
	 &=&  \hat{x}_k+\frac{k+\alpha-2}{\alpha-1}(x_{k+1}-\bar{x}_k)\nonumber
\end{eqnarray}
and
\begin{equation}\label{eq:th1_4}
	 \hat{x}_{k+1}-x^* = x_{k+1}-x^*+\frac{k-1}{\alpha-1}(x_{k+1}-x_k) .
\end{equation}
Similarly, we have
\begin{equation}\label{eq:th1_5}
	  \hat{\lambda}_{k+1} = \hat{\lambda}_k+\frac{k+\alpha-2}{\alpha-1}(\lambda_{k+1}-\bar{\lambda}_k)
\end{equation}
and
\begin{equation}\label{eq:th1_6}
	\hat{\lambda}_{k+1}-\lambda^* = \lambda_{k+1}-\lambda^*+\frac{k-1}{\alpha-1}(\lambda_{k+1}-\lambda_k).
\end{equation}

By the definition of $\mathcal{L}(x,\lambda)$,   we get
$\partial_x \mathcal{L}(x,\lambda) = \partial F(x)+A^T\lambda$. Combining this and equality \eqref{eq:th1_6}, we can rewrite \eqref{sq:sequ1} as
\begin{eqnarray*}
	\frac{k+\alpha-2}{k}M_k(x_{k+1}-\bar{x}_k)&\in&  - s(\partial F(x_{k+1})+A^T\lambda^*+A^T(\lambda_{k+1}-\lambda^*+\frac{k-1}{\alpha-1}(\lambda_{k+1}-\lambda_{k})))\\
	& =&  - s\partial_x \mathcal{L}(x_{k+1},\lambda^*)-sA^T(\hat{\lambda}_{k+1}-\lambda^*),
\end{eqnarray*}
which implies
\begin{equation}\label{eq:th1_7}
	\xi_k :=-\frac{k+\alpha-2}{sk}M_k(x_{k+1}-\bar{x}_k)- A^T(\hat{\lambda}_{k+1}-\lambda^*)\in \partial_x \mathcal{L}(x_{k+1},\lambda^*).
\end{equation}
Since $M_{k-1}\succcurlyeq M_k\succcurlyeq 0$, it follows from \eqref{eq:known_2} and \eqref{eq:th1_3} that
\begin{eqnarray}\label{eq:th1_8}
	&&\frac{1}{2}\|\hat{x}_{k+1}-x^*\|_{M_k}^2-\frac{1}{2}\|\hat{x}_k-x^*\|_{M_{k-1}}^2\nonumber \\
	&&\qquad\quad =\frac{1}{2}\|\hat{x}_{k+1}-x^*\|_{M_k}^2-\frac{1}{2}\|\hat{x}_k-x^*\|_{M_{k}}^2-\frac{1}{2}\|\hat{x}_k-x^*\|_{M_{k-1}-M_k}^2\nonumber \\
	&& \qquad\quad\leq \langle \hat{x}_{k+1}-x^*, M_k(\hat{x}_{k+1}-\hat{x}_k)\rangle-\frac{1}{2}\|\hat{x}_{k+1}-\hat{x}_k\|_{M_k}^2 \nonumber \\
	&&\qquad\quad = \frac{k+\alpha-2}{\alpha-1}\langle \hat{x}_{k+1}-x^*, M_k(x_{k+1}-\bar{x}_k)\rangle-\frac{(k+\alpha-2)^2}{2(\alpha-1)^2}\|x_{k+1}-\bar{x}_k\|_{M_k}^2 \\
	&&\qquad\quad =  -\frac{sk}{\alpha-1}(\langle \hat{x}_{k+1}-x^*, \xi_k\rangle  +\langle \hat{x}_{k+1}-x^*,A^T(\hat{\lambda}_{k+1}-\lambda^*)\rangle)\nonumber\\
	&&\qquad\qquad -\frac{(k+\alpha-2)^2}{2(\alpha-1)^2}\|x_{k+1}-\bar{x}_k\|_{M_k}^2.\nonumber
\end{eqnarray}
Since $\mathcal{L}(x,\lambda^*)$ is a convex function with respect to $x$, from \eqref{eq:th1_4} and \eqref{eq:th1_7} we get
\begin{eqnarray}\label{eq:th1_9}
	\langle \hat{x}_{k+1}-x^*, \xi_k \rangle &=&\langle x_{k+1}-x^*, \xi_k\rangle +\frac{k-1}{\alpha-1}\langle x_{k+1}-x_k, \xi_k\rangle\nonumber\\
	&\geq&  \mathcal{L}(x_{k+1},\lambda^*)-\mathcal{L}(x^*,\lambda^*)+\frac{k-1}{\alpha-1}(\mathcal{L}(x_{k+1},\lambda^*)-\mathcal{L}(x_{k},\lambda^*)).
\end{eqnarray}
Combining \eqref{eq:th1_8}  and \eqref{eq:th1_9} together, we have
\begin{eqnarray}\label{eq:th1_10}
	&&\frac{1}{2}\|\hat{x}_{k+1}-x^*\|_{M_k}^2-\frac{1}{2}\|\hat{x}_k-x^*\|_{M_{k-1}}^2\nonumber\\
	&&\quad \leq -\frac{sk}{\alpha-1}(\mathcal{L}(x_{k+1},\lambda^*)-\mathcal{L}(x^*,\lambda^*)) -  \frac{s(k^2-k)}{(\alpha-1)^2}(\mathcal{L}(x_{k+1},\lambda^*)-\mathcal{L}(x_{k},\lambda^*))\nonumber  \\
	&&\qquad -\frac{sk}{\alpha-1}\langle \hat{x}_{k+1}-x^*,A^T(\hat{\lambda}_{k+1}-\lambda^*)\rangle-\frac{(k+\alpha-2)^2}{2(\alpha-1)^2}\|x_{k+1}-\bar{x}_k\|_{M_k}^2.
\end{eqnarray}
Since $Ax^*=b$, it follows from Step 3 of Algorithm \ref{al:al1} and \eqref{eq:th1_4} that
\begin{equation*}\label{eq:th1_11}
	 \lambda_{k+1}-\bar{\lambda}_k=\frac{sk}{k+\alpha-2}(Ax_{k+1}-Ax^*+\frac{k-1}{\alpha-1}A(x_{k+1}-x_k)) = \frac{sk}{k+\alpha-2}A(\hat{x}_{k+1}-x^*).
\end{equation*}
This together with \eqref{eq:known_2} and \eqref{eq:th1_5} yields
\begin{eqnarray}\label{eq:th1_12}
	&&\frac{1}{2}\|\hat{\lambda}_{k+1}-\lambda^*\|^2-\frac{1}{2}\|\hat{\lambda}_k-\lambda^*\|^2 = \langle \hat{\lambda}_{k+1}-\lambda^*, \hat{\lambda}_{k+1}-\hat{\lambda}_k\rangle-\frac{1}{2}\|\hat{\lambda}_{k+1}-\hat{\lambda}_k\|^2 \nonumber\\
	&&\qquad\qquad = \frac{k+\alpha-2}{\alpha-1}\langle \hat{\lambda}_{k+1}-\lambda^*,\lambda_{k+1}-\bar{\lambda}_k
\rangle-\frac{(k+\alpha-2)^2}{2(\alpha-1)^2}\|\lambda_{k+1}-\bar{\lambda}_k\|^2\\
 	 &&\qquad\qquad = \frac{sk}{\alpha-1}\langle \hat{\lambda}_{k+1}-\lambda^*,A(\hat{x}_{k+1}-x^*)\rangle-\frac{(k+\alpha-2)^2}{2(\alpha-1)^2}\|\lambda_{k+1}-\bar{\lambda}_k\|^2.\nonumber
\end{eqnarray}
It follows from \eqref{eq:th1_10} and \eqref{eq:th1_12} that
\begin{eqnarray*}\label{eq:th1_14}
	&&\mathcal{E}_{k+1}-\mathcal{E}_k\nonumber \\
	&&\qquad =\frac{s(k^2+k)}{(\alpha-1)^2}(\mathcal{L}(x_{k+1},\lambda^*)-\mathcal{L}(x^*,\lambda^*))- \frac{s(k^2-k)}{(\alpha-1)^2}(\mathcal{L}(x_k,\lambda^*)-\mathcal{L}(x^*,\lambda^*))\nonumber\\	
	&&\qquad\quad  +\frac{1}{2}\|\hat{x}_{k+1}-x^*\|_{M_k}^2-\frac{1}{2}\|\hat{x}_k-x^*\|_{M_{k-1}}^2+\frac{1}{2}\|\hat{\lambda}_{k+1}-\lambda^*\|^2-\frac{1}{2}\|\hat{\lambda}_k-\lambda^*\|^2\\
	&&\qquad \leq \frac{(3-\alpha)sk}{(\alpha-1)^2}(\mathcal{L}(x_{k+1},\lambda^*)-\mathcal{L}(x^*,\lambda^*)) -\frac{(k+\alpha-2)^2}{2(\alpha-1)^2}(\|x_{k+1}-\bar{x}_k\|_{M_k}^2+\|\lambda_{k+1}-\bar{\lambda}_k\|^2)\nonumber\\
	&&\qquad\leq    -\frac{(k+\alpha-2)^2}{2(\alpha-1)^2}(\|x_{k+1}-\bar{x}_k\|_{M_k}^2+\|\lambda_{k+1}-\bar{\lambda}_k\|^2),\nonumber
\end{eqnarray*}
where the last inequality follows from $\alpha\geq 3$ and $(x^*,\lambda^*)\in\Omega$. This yields the desire result.
\end{proof}

To obtain the fast convergence rates, we need the following lemma.
\begin{lemma}(\cite[Lemma 2]{Li2019}, \cite[Lemma 3.18]{Lin2020})\label{le_new3}
	Let $\{a_k\}_{k=1}^{+\infty}$ be a sequence of vectors such that 
	\[\|(\tau+(\tau-1)K)a_{K+1}+\sum^{K}_{k=1}a_k\|\leq 	C,\qquad  \forall K\geq 1, \]
where $\tau>1$ and $C\geq 0$. Then $\|\sum^{K}_{k=1}a_k\|\leq C$  for all $K\geq 1$.
\end{lemma}

Now, we discuss the $\mathcal{O}(1/k^2)$ convergence rate of  Algorithms \ref{al:al1}.
\begin{theorem}\label{th:cov_al1}
	Suppose that $F$ is a  closed convex function, $\Omega\neq\emptyset$ and $M_{k-1}\succcurlyeq M_k$. Let $\{(x_k,\lambda_k,\bar{x}_k,\\
	\bar{\lambda}_k)\}_{k\geq 1}$ be the sequence generated by Algorithm \ref{al:al1} and  $(x^*,\lambda^*)\in\Omega$. The following conclusions hold:
\begin{itemize}
\item[(i)]$\sum_{k=1}^{+\infty}k^2(\|x_{k+1}-\bar{x}_k\|_{M_k}^2+\|\lambda_{k+1}-\bar{\lambda}_k\|^2) < +\infty$.
	\item[(ii)] For all $k> 1$,
	\begin{eqnarray*}
		&&\|Ax_{k}-b\| \leq  \frac{4(\alpha-1)^2\sqrt{2\mathcal{E}_{1}}}{s(k-1)(k+\alpha-3)},\\
		&&|F(x_k)-F(x^*)|\leq  \frac{(\alpha-1)^2\mathcal{E}_{1}}{s(k^2-k)}+ \frac{4(\alpha-1)^2\sqrt{2\mathcal{E}_{1}}\|\lambda^*\|}{s(k-1)(k+\alpha-3)},
	\end{eqnarray*}
where  $\mathcal{E}_{1}=\frac{1}{2}\|x_1-x^*\|_{M_0}^2+\frac{1}{2}\|\lambda_1-\lambda^*\|^2$. 
\end{itemize}
\end{theorem}
\begin{proof}
From Lemma \ref{le_new2}, we have
\begin{equation}\label{eq:th1_16}
	\mathcal{E}_{k+1}- \mathcal{E}_k \leq  -\frac{(k+\alpha-2)^2}{2(\alpha-1)^2}(\|x_{k+1}-\bar{x}_k\|_{M_k}^2+\|\lambda_{k+1}-\bar{\lambda}_k\|^2)\leq 0, \quad \forall k\geq 1.
\end{equation}
By the definition of $\mathcal{E}_{k}$ and \eqref{eq:th1_16},  $\{\mathcal{E}_{k}\}_{k\geq 1}$ is a nonincreasing and positive sequence.
As a consequence, $\mathcal{E}_{k}$  converges  to some point.  It follows from  \eqref{eq:th1_14} that
\begin{eqnarray}\label{eq:th1_17}
	&&\sum_{k=1}^{+\infty} \frac{(k+\alpha-1)^2}{2(\alpha-1)^2}(\|x_{k+1}-\bar{x}_k\|_{M_k}^2+\|\lambda_{k+1}-\bar{\lambda}_k\|^2))\nonumber\\
	&&\quad \leq \lim_{K\to+\infty} \sum_{k=1}^K (\mathcal{E}_{k}-\mathcal{E}_{k+1})\nonumber \\
	&&\quad = \mathcal{E}_{1}-\lim_{K\to+\infty}\mathcal{E}_{K+1}\\
	&& \quad< +\infty,\nonumber
\end{eqnarray}
which is (i).

Combining \eqref{energy1} and \eqref{eq:th1_16}, we have
\begin{eqnarray*}
  \|\hat{\lambda}_k-\lambda^*\| \leq  \sqrt{2\mathcal{E}_{k}}\leq \sqrt{2\mathcal{E}_{1}}, \qquad \forall k\geq 1.
\end{eqnarray*}
This yields
\begin{eqnarray}\label{eq_new1}
  \left\|\sum^K_{k=1}(\hat{\lambda}_{k+1}-\hat{\lambda}_{k})\right\| &=&   \left\|\hat{\lambda}_{K+1}-\hat{\lambda}_{1}\right\| \nonumber \\
  &\leq& \left\|\hat{\lambda}_{K+1}-{\lambda}^*\right\|+\left\|\hat{\lambda}_{1}-{\lambda}^*\right\|\\
  &\leq& 2\sqrt{2\mathcal{E}_{1}}\nonumber
\end{eqnarray}
for all $K\geq 1$.
It follows form \eqref{eq:th1_5} and Step 3 of Algorithm \ref{al:al1} that
\begin{eqnarray*}
	 \left\|\sum^K_{k=1}(\hat{\lambda}_{k+1}-\hat{\lambda}_{k})\right\| &=& \left\|\sum^K_{k=1}\frac{k+\alpha-2}{\alpha-1}(\lambda_{k+1}-\bar{\lambda}_k)\right\|\\
	 &=& \frac{s}{\alpha-1}\left\|\sum^K_{k=1}k(Ax_{k+1}-b+\frac{k-1}{\alpha-1}A(x_{k+1}-x_k))\right\|\\
	 & =&\frac{s}{(\alpha-1)^2}\left\|\sum^K_{k=1}\left(k(k+\alpha-2)(Ax_{k+1}-b) -k(k-1)(Ax_{k}-b)    \right)\right\|\\
	 & =&\frac{s}{(\alpha-1)^2}\left\|K(K+\alpha-2)(Ax_{K+1}-b)+ \sum^K_{k=1}(\alpha-3)(k-1)(Ax_{k}-b)\right\|.
\end{eqnarray*}
This together with \eqref{eq_new1} implies
\begin{eqnarray}\label{eq_new2}
	\left\|K(K+\alpha-2)(Ax_{K+1}-b)+ \sum^K_{k=1}\left((\alpha-3)(k-1)(Ax_{k}-b)\right)\right\|	\leq \frac{2(\alpha-1)^2\sqrt{2\mathcal{E}_{1}}}{s}.
\end{eqnarray}

When $\alpha = 3$, it follows from \eqref{eq_new2} that
 \begin{eqnarray*}
 	\|K(K+\alpha-2)(Ax_{K+1}-b)\| \leq  \frac{2(\alpha-1)^2\sqrt{2\mathcal{E}_{1}}}{s}.
 \end{eqnarray*}
When $\alpha>3$: applying  Lemma \ref{le_new3} with $a_k= (\alpha-3)(k-1)(Ax_{k}-b)$, $\tau = \frac{\alpha-2}{\alpha-3}$ and $C = \frac{2(\alpha-1)^2\sqrt{2\mathcal{E}_{1}}}{s}$, from \eqref{eq_new2}, we obtain
\[\sum^K_{k=1}\left((\alpha-3)(k-1)(Ax_{k}-b)\right) \leq \frac{2(\alpha-1)^2\sqrt{2\mathcal{E}_{1}}}{s},\]
which together with \eqref{eq_new2} yields
\[ \|K(K+\alpha-2)(Ax_{K+1}-b)\|\leq \frac{4(\alpha-1)^2\sqrt{2\mathcal{E}_{1}}}{s}.\]
From above discussion, when $\alpha\geq 3$, we have
\begin{eqnarray}\label{eq_new3}
	\|Ax_{k}-b\| \leq  \frac{4(\alpha-1)^2\sqrt{2\mathcal{E}_{1}}}{s(k-1)(k+\alpha-3)}, \quad \forall k>1.
\end{eqnarray}
It follows form the definition of $\mathcal{E}_{k}$ and \eqref{eq:th1_16} that
\begin{eqnarray*}
	\mathcal{L}(x_k,\lambda^*)-\mathcal{L}(x^*,\lambda^*) \leq \frac{(\alpha-1)^2\mathcal{E}_{k}}{s(k^2-k)}\leq  \frac{(\alpha-1)^2\mathcal{E}_{1}}{s(k^2-k)}
\end{eqnarray*}
for all $k>1$. This together with \eqref{eq_new3} implies that
\begin{eqnarray*}
	|F(x_k)-F(x^*)| &=& |\mathcal{L}(x_k,\lambda^*)-\mathcal{L}(x^*,\lambda^*)-\langle \lambda^*, Ax_k-b\rangle|\\
	&\leq & \mathcal{L}(x_k,\lambda^*)-\mathcal{L}(x^*,\lambda^*) +\|\lambda^*\|\|Ax_k-b\|\\
	&\leq &  \frac{(\alpha-1)^2\mathcal{E}_{1}}{s(k^2-k)}+ \frac{4(\alpha-1)^2\sqrt{2\mathcal{E}_{1}}\|\lambda^*\|}{s(k-1)(k+\alpha-3)}
\end{eqnarray*}
for all $k> 1$.
The proof  is complete.
\end{proof}

To investigate the convergence of  Algorithm \ref{al:al2}, we  need the following assumption.

{\bf Assumption (H):} $\Omega\neq\emptyset$, $f$ is a closed convex function, and $g$ is a convex  smooth function  and has a Lipschitz continuous gradient with constant $L_g> 0$, i.e.,
\begin{equation*}\label{eq:ass_1}
	\|\nabla g(x)-\nabla g(y)\|\leq L_g\|x-y\|,\qquad \forall x,y\in\mathbb{R}^n,
\end{equation*}
equivalently,
\begin{equation}\label{eq:ass_2}
	g(x)\leq g(y) + \langle \nabla g(y),x-y\rangle+\frac{L_g}{2}\|x-y\|^2,\qquad \forall x,y\in\mathbb{R}^n.
\end{equation}

\begin{lemma}\label{pro:al2}
	Let $\{(x_k,\lambda_k,\bar{x}_k)\}_{k\geq 1}$ the  sequence generated by Algorithm \ref{al:al2}. Then
\begin{eqnarray}\label{sq:al2_1}
	  \frac{k+\alpha-2}{k} M_k({x_{k+1}-\bar{x}_k})\in  - s\left(\partial f(x_{k+1})+\nabla g(\bar{x}_k)+ A^T(\lambda_{k+1}+\frac{k-1}{\alpha-1}(\lambda_{k+1}-\lambda_{k}))-\epsilon_k\right).
\end{eqnarray}
\end{lemma}
\begin{proof}
 From Step 2 of Algorithm \ref{al:al2} , we have
 	 \[ 0\in \partial f(x_{k+1})+\nabla g(\bar{x}_k)+ \frac{k+\alpha-2}{sk}M_k(x_{k+1}-\bar{x}_k)+\frac{sk(k+\alpha-2)}{(\alpha-1)^2} A^T(Ax_{k+1}-\eta_k)+A^T\hat{\lambda}_k-\epsilon_k ,\]
which yields
 \begin{equation*}
 	 \frac{k+\alpha-2}{k}M_k(x_{k+1}-\bar{x}_k) \in -s\left(\partial f(x_{k+1})+\nabla g(\bar{x}_k) +A^T\left( \frac{sk(k+\alpha-2)}{(\alpha-1)^2} (Ax_{k+1}-\eta_k)+\hat{\lambda}_k\right)-\epsilon_k\right).
 \end{equation*}
The rest of the proof is similar as the one of Lemma \ref{pro:al1}, and so we omit it.
  \end{proof}

\begin{lemma}\label{le_new5}
	Assume that Assumption (H) holds, and $M_{k-1}\succcurlyeq M_k\succcurlyeq sL_g Id $.  Let $\{(x_k,\lambda_k, \hat{\lambda}_k)\}_{k\geq 1}$ be the sequence generated by Algorithm \ref{al:al2} and $(x^*,\lambda^*)\in\Omega$. Define
\begin{equation}\label{eq:th3_1}
	\mathcal{E}^{\epsilon}_k = \mathcal{E}_k-\sum^{k}_{j=1} \frac{s(j-1)}{\alpha-1} \langle \hat{x}_j-x^*, \epsilon_{j-1}\rangle,
\end{equation}
where $\mathcal{E}_k$ is defined in \eqref{energy1} and $\hat{x}_k$ is defined  in \eqref{hat_x}. Then, for any $k\geq 1$,
\begin{eqnarray*}
	\mathcal{E}^{\epsilon}_{k+1} &\leq& \mathcal{E}^{\epsilon}_k.
\end{eqnarray*}
\end{lemma}
\begin{proof}
	
 By same arguments as in the proof of Lemma \ref{le_new2}, we get
\begin{eqnarray}
	&& \hat{x}_{k+1}-\hat{x}_k = \frac{k+\alpha-2}{\alpha-1}(x_{k+1}-\bar{x}_k),\label{eq:th2_2}\\
	&& \hat{x}_{k+1}-x^* = x_{k+1}-x^*+\frac{k-1}{\alpha-1}(x_{k+1}-x_k),\label{eq:th2_3}\\
	&& \hat{\lambda}_{k+1}-\hat{\lambda}_k=\frac{k+\alpha-2}{\alpha-1}(\lambda_{k+1}-\bar{\lambda}_k),\label{eq:th2_4}\\
	&& \hat{\lambda}_{k+1}-\lambda^* = \lambda_{k+1}-\lambda^*+\frac{k-1}{\alpha-1}(\lambda_{k+1}-\lambda_k).\label{eq:th2_5}
\end{eqnarray}
For notation simplicity, we denote
\begin{equation}
	\mathcal{L}^{f}(x) = f(x)+\langle \lambda^*,Ax-b\rangle.
\end{equation}
Then  $\mathcal{L}^{f}$ is a convex function, $\partial \mathcal{L}^{f}(x) = \partial f(x)+A^T\lambda^*$, and
\begin{equation}\label{eq:th2_extra1}
\mathcal{L}(x,\lambda^*) = \mathcal{L}^{f}(x)+g(x).
\end{equation}
It follows from \eqref{sq:al2_1} and \eqref{eq:th2_5} that
\begin{eqnarray*}
	  \frac{k+\alpha-2}{k} M_k(x_{k+1}-\bar{x}_k)&\in& - s(\partial f(x_{k+1})+A^T\lambda^*)-s\nabla g(\bar{x}_k)\\
	&&-s A^T(\lambda_{k+1}-\lambda^*+\frac{k-1}{\alpha-1}(\lambda_{k+1}-\lambda_{k}))+s\epsilon_k \\
	& =&  - s\partial \mathcal{L}^{f}(x_{k+1})-s\nabla g(\bar{x}_k)-sA^T(\hat{\lambda}_{k+1}-\lambda^*)+s\epsilon_k,
\end{eqnarray*}
which yields
\begin{equation}\label{eq:th2_6}
	\xi_k :=-\frac{k+\alpha-2}{ks}M_k(x_{k+1}-\bar{x}_k)-\nabla g(\bar{x}_k)- A^T(\hat{\lambda}_{k+1}-\lambda^*)+\epsilon_k \in \partial\mathcal{L}^{f}(x_{k+1}).
\end{equation}
Since $M_{k-1}\succcurlyeq M_k$, it follows from \eqref{eq:known_2}, \eqref{eq:th2_2} and \eqref{eq:th2_6} that
\begin{eqnarray}\label{eq:th2_7}
	&&\frac{1}{2}\|\hat{x}_{k+1}-x^*\|_{M_k}^2-\frac{1}{2}\|\hat{x}_k-x^*\|_{M_{k-1}}^2\leq \langle \hat{x}_{k+1}-x^*, M_k(\hat{x}_{k+1}-\hat{x}_k)\rangle-\frac{1}{2}\|\hat{x}_{k+1}-\hat{x}_k\|^2_{M_k} \nonumber \\
	&&\qquad =  -\frac{sk}{\alpha-1}(\langle \hat{x}_{k+1}-x^*,\xi_k\rangle+\langle \hat{x}_{k+1}-x^*,A^T(\hat{\lambda}_{k+1}-\lambda^*)\rangle\\
	&&\qquad\quad  +\langle \hat{x}_{k+1}-x^*,\nabla g(\bar{x}_k)\rangle-\langle \hat{x}_{k+1}-x^*,\epsilon_k\rangle) -\frac{(k+\alpha-2)^2}{2(\alpha-1)^2}\|x_{k+1}-\bar{x}_k\|_{M_k}^2.\nonumber
\end{eqnarray}
From \eqref{eq:th2_3} and \eqref{eq:th2_6}, we have
\begin{eqnarray}\label{eq:th2_8}
	\langle \hat{x}_{k+1}-x^*, \xi_k \rangle& =&\langle x_{k+1}-x^*, \xi_k\rangle +\frac{k-1}{\alpha-1}\langle x_{k+1}-x_k, \xi_k\rangle\nonumber\\
	&\geq&  \mathcal{L}^f(x_{k+1})-\mathcal{L}^f(x^*)+\frac{k-1}{\alpha-1}(\mathcal{L}^f(x_{k+1})-\mathcal{L}^f(x_{k})),
\end{eqnarray}
where the inequality follows from the convexity of $\mathcal{L}^f$.  Since $g$ has a Lipschitz continuous gradient, from \eqref{eq:ass_2} we get
\begin{equation}\label{eq:th2_9}
	g(x_{k+1})\leq g(\bar{x}_k)+\langle \nabla g(\bar{x}_k),x_{k+1}-\bar{x}_k\rangle+\frac{L_g}{2}\|x_{k+1}-\bar{x}_k\|^2.
\end{equation}
By the convexity of $g$, we have
\begin{eqnarray}\label{eq:th2_10}
	\langle\nabla g(\bar{x}_k),x_{k+1}-\bar{x}_k\rangle &=& \langle\nabla g(\bar{x}_k),x_{k+1}-x^*\rangle+\langle\nabla g(\bar{x}_k),x^*-\bar{x}_k\rangle\nonumber\\
	&\leq&\langle\nabla g(\bar{x}_k),x_{k+1}-x^*\rangle + g(x^*)-g(\bar{x}_k)
\end{eqnarray}
and
\begin{eqnarray}\label{eq:th2_11}
\langle\nabla g(\bar{x}_k),x_{k+1}-\bar{x}_k\rangle &=& \langle\nabla g(\bar{x}_k),x_{k+1}-x_k\rangle+\langle\nabla g(\bar{x}_k),x_k-\bar{x}_k\rangle\nonumber\\
	&\leq&\langle\nabla g(\bar{x}_k),x_{k+1}-x_k\rangle + g(x_k)-g(\bar{x}_k).
\end{eqnarray}
It follows from  \eqref{eq:th2_9}-\eqref{eq:th2_11} that
\[\langle\nabla g(\bar{x}_k),x_{k+1}-x^*\rangle\geq g(x_{k+1})-g(x^*)-\frac{L_g}{2}\|x_{k+1}-\bar{x}_k\|^2,\]
and
\[\langle\nabla g(\bar{x}_k),x_{k+1}-x_k\rangle\geq g(x_{k+1})-g(x_k)-\frac{L_g}{2}\|x_{k+1}-\bar{x}_k\|^2.\]
This together with \eqref{eq:th2_3} yields
\begin{eqnarray}\label{eq:th2_12}
	\langle \hat{x}_{k+1}-x^*,\nabla g(\bar{x}_k) \rangle &=& \langle\nabla g(\bar{x}_k),x_{k+1}-x^*\rangle+\frac{k-1}{\alpha-1}\langle\nabla g(\bar{x}_k),x_{k+1}-x_k\rangle\nonumber\\
	& \geq& g(x_{k+1})-g(x^*)+\frac{k-1}{\alpha-1}(g(x_{k+1})-g(x_k))\\
	&&-\frac{(k+\alpha-2)L_g}{2(\alpha-1)}\|x_{k+1}-\bar{x}_k\|^2.\nonumber
\end{eqnarray}
It follows from  \eqref{eq:th2_7},\eqref{eq:th2_8} and \eqref{eq:th2_12} that
\begin{eqnarray}\label{eq:th2_13}
	&&\frac{1}{2}\|\hat{x}_{k+1}-x^*\|_{M_k}^2-\frac{1}{2}\|\hat{x}_k-x^*\|_{M_{k-1}}^2\nonumber\\
	&&\qquad\qquad \leq -\frac{sk}{\alpha-1}(\mathcal{L}^f(x_{k+1})+g(x_{k+1})-(\mathcal{L}^f(x^*)+g(x^*)))\nonumber \\
	&&\qquad\qquad\quad  -  \frac{sk(k-1)}{(\alpha-1)^2}(\mathcal{L}^f(x_{k+1})+g(x_{k+1})-(\mathcal{L}^f(x_{k})+g(x_k)))\nonumber\\
	&&\qquad\qquad\quad -\frac{sk}{\alpha-1}\langle \hat{x}_{k+1}-x^*,A^T(\hat{\lambda}_{k+1}-\lambda^*)\rangle+ \frac{sk}{\alpha-1}\langle \hat{x}_{k+1}-x^*,\epsilon_k \rangle\\
	&&\qquad\qquad\quad-\frac{k+\alpha-2}{2(\alpha-1)^2}\|x_{k+1}-\bar{x}_k\|_{(k+\alpha-2)M_k-sL_gkId}^2 \nonumber\\
	&&\qquad\qquad \leq -\frac{sk}{\alpha-1}(\mathcal{L}(x_{k+1},\lambda^*)-\mathcal{L}(x^*,\lambda^*))- \frac{sk(k-1)}{(\alpha-1)^2}(\mathcal{L}(x_{k+1},\lambda^*)-\mathcal{L}(x_k,\lambda^*))\nonumber\\
	&&\qquad\qquad\quad -\frac{sk}{\alpha-1}\langle \hat{x}_{k+1}-x^*,A^T(\hat{\lambda}_{k+1}-\lambda^*)\rangle+\frac{sk}{\alpha-1}\langle \hat{x}_{k+1}-x^*,\epsilon_k \rangle, \nonumber
\end{eqnarray}
where the second inequality follows from the assumption $M_k\succcurlyeq sL_g Id\succcurlyeq \frac{sL_gk}{k+\alpha-2}Id$.

It follows from  \eqref{eq:th1_12}, \eqref{eq:th3_1} and \eqref{eq:th2_13}  that
\begin{eqnarray}\label{eq:th2_16}
	\mathcal{E}^{\epsilon}_{k+1}-\mathcal{E}^{\epsilon}_k&=& \mathcal{E}_{k+1}-\mathcal{E}_k- \frac{sk}{\alpha-1} \langle \hat{x}_{k+1}-x^*, \epsilon_k\rangle\nonumber \\
	& \leq& \frac{(3-\alpha)sk}{(\alpha-1)^2}(\mathcal{L}(x_{k+1},\lambda^*)-\mathcal{L}(x^*,\lambda^*))  -\frac{(k+\alpha-2)^2}{2(\alpha-1)^2}\|\lambda_{k+1}-\bar{\lambda}_k\|^2\nonumber \\
	& \leq& 0,
\end{eqnarray}
where the last inequality follows from $\alpha\geq 3$ and $(x^*,\lambda^*)\in\Omega$.
\end{proof}

To analyze the convergence of Algorithm \ref{al:al2}, we need the following discrete version of the Gronwall-Bellman lemma.

\begin{lemma}\cite[Lemma 5.14]{AttouchCPR2018} \label{le:Gronwall_Bellman}
	Let $\{a_k\}_{k\geq 1}$ and $\{b_k\}_{k\geq 1}$ be two nonnegative sequences  such that $\sum^{+\infty}_k b_k< +\infty$ and
	\[ a_k^2 \leq c^2 +\sum^{k}_{j=1}b_j a_j \]
	for all $k\geq 1$, where $c\geq 0$. Then
	\[a_k\leq c+\sum^{+\infty}_{j=1} b_j\]
	for all  $k\geq 1$.
\end{lemma}

\begin{theorem}\label{th:cov_al3}
	Assume that Assumption (H) holds,  $M_{k-1}\succcurlyeq M_k\succcurlyeq sL_gId$  for all $k\geq 1$, and
\[\sum^{+\infty}_{k=1}k\|\epsilon_k\|< +\infty. \]
Let $\{(x_k,\lambda_k)\}_{k\geq 1}$ be the sequence generated by Algorithm \ref{al:al2} and $(x^*,\lambda^*)\in\Omega$. Then for all $k> 1$,
\begin{eqnarray*}
	&& \|Ax_{k}-b\| \leq  \frac{4(\alpha-1)^2\sqrt{2C}}{s(k-1)(k+\alpha-3)},\\
	&& |f(x_k)+g(x_k)-f(x^*)-g(x^*)| \leq   \frac{(\alpha-1)^2C}{s(k^2-k)}+\frac{4(\alpha-1)^2\sqrt{2C}\|\lambda^*\|}{s(k-1)(k+\alpha-3)},
\end{eqnarray*}
where
\[ C: =  \frac{1}{2}\|x_1-x^*\|_{M_0}^2+\frac{1}{2}\|\lambda_1-\lambda^*\|^2+\frac{s}{\alpha-1}\left(\sqrt{\frac{2\mathcal{E}_{1}}{sL_g}}+\frac{2}{(\alpha-1)L_g}\sum^{+\infty}_{j=1}j\|\epsilon_{j}\|\right )\times\sum^{+\infty}_{j=1}j\|\epsilon_{j}\|.\]
\end{theorem}
\begin{proof}
From Lemma \ref{le_new5} we have
\[\mathcal{E}^{\epsilon}_{k+1} \leq \mathcal{E}^{\epsilon}_k\leq  \mathcal{E}^{\epsilon}_1,\]
and it yields
 \begin{equation}\label{eq:th3_3}
	\mathcal{E}_{k}\leq \mathcal{E}_{1}+\sum^{k}_{j=1}\frac{s(j-1)}{\alpha-1} \langle \hat{x}_j-x^*, \epsilon_{j-1}\rangle.
\end{equation}
This together with \eqref{energy1} and Cauchy-Schwarz inequality implies
\begin{equation*}\|\hat{x}_k-x^*\|_{M_{k-1}}^2\leq 2\mathcal{E}_{1}+\frac{2s}{\alpha-1}\sum^{k}_{j=1}(j-1) \|\hat{x}_j-x^*\|\cdot\|\epsilon_{j-1}\| .
\end{equation*}
Since $M_{k-1}\succcurlyeq sL_g Id $,
\begin{equation}\label{eqfyp3}
\|\hat{x}_k-x^*\|^2\leq \frac{1}{sL_g} \|\hat{x}_k-x^*\|_{M_{k-1}}^2\leq \frac{2\mathcal{E}_{1}}{sL_g}+\frac{2}{(\alpha-1)L_g}\sum^{k}_{j=1}(j-1) \|\hat{x}_j-x^*\|\cdot\|\epsilon_{j-1}\| .
\end{equation}
Since $\sum^{+\infty}_{j=1} j\|\epsilon_j\|< +\infty$, applying Lemma \ref{le:Gronwall_Bellman} with $a_k=\|\hat{x}_k-x^*\|$ to \eqref{eqfyp3}, we obtain
\begin{equation}\label{eq:th3_4}
	\|\hat{x}_k-x^*\|\leq \sqrt{\frac{2\mathcal{E}_{1}}{sL_g}}+\frac{2}{(\alpha-1)L_g}\sum^{+\infty}_{j=1}j\|\epsilon_{j}\|<+\infty, \quad\forall k\geq 1.
\end{equation}
This together with \eqref{eq:th3_3} yields
\begin{eqnarray}\label{ineqfyp5}
	\mathcal{E}_{k}&\leq & \mathcal{E}_{1}+\frac{s}{\alpha-1}\sup_{k\geq 1}\|\hat{x}_k-x^*\|\times\sum^{+\infty}_{j=1}j\|\epsilon_j\|\nonumber\\
	&\leq& \mathcal{E}_{1}+\frac{s}{\alpha-1}\left(\sqrt{\frac{2\mathcal{E}_{1}}{sL_g}}+\frac{2}{(\alpha-1)L_g}\sum^{+\infty}_{j=1}j\|\epsilon_{j}\|\right )\times\sum^{+\infty}_{j=1}j\|\epsilon_{j}\|
\end{eqnarray}
for any $k\geq 1$.
Denote
\begin{eqnarray*}
	&&C:=  \mathcal{E}_{1}+\frac{s}{\alpha-1}\left(\sqrt{\frac{2\mathcal{E}_{1}}{sL_g}}+\frac{2}{(\alpha-1)L_g}\sum^{+\infty}_{j=1}j\|\epsilon_{j}\|\right )\times\sum^{+\infty}_{j=1}j\|\epsilon_{j}\|\\
	&&\quad\ = \frac{1}{2}\|x_1-x^*\|_{M_0}^2+\frac{1}{2}\|\lambda_1-\lambda^*\|^2+\frac{s}{\alpha-1}\left(\sqrt{\frac{2\mathcal{E}_{1}}{sL_g}}+\frac{2}{(\alpha-1)L_g}\sum^{+\infty}_{j=1}j\|\epsilon_{j}\|\right )\times\sum^{+\infty}_{j=1}j\|\epsilon_{j}\|.
\end{eqnarray*}
 By the definition of $\mathcal{E}_{k}$ and \eqref{ineqfyp5}, we have
\[\frac{s(k^2-k)}{(\alpha-1)^2}(\mathcal{L}(x_k,\lambda^*)-\mathcal{L}(x^*,\lambda^*))\leq \mathcal{E}_{k}\leq C\]
and
\begin{eqnarray*}
  \|\hat{\lambda}_k-\lambda^*\| \leq  \sqrt{2\mathcal{E}_{k}}\leq \sqrt{2C}, \qquad \forall k\geq 1.
\end{eqnarray*}
By similar arguments in Theorem \ref{th:cov_al1}, we obtain
\begin{eqnarray*}
	\|Ax_{k}-b\| \leq  \frac{4(\alpha-1)^2\sqrt{2C}}{s(k-1)(k+\alpha-3)}
\end{eqnarray*}
and
\begin{eqnarray*}
	&&|f(x_k)+g(x_k)-f(x^*)-g(x^*)| \\
	&&\qquad\qquad \leq  \mathcal{L}(x_k,\lambda^*)-\mathcal{L}(x^*,\lambda^*) +\|\lambda^*\|\|Ax_k-b\|\\
	&&\qquad\qquad \leq   \frac{(\alpha-1)^2C}{s(k^2-k)}+\frac{4(\alpha-1)^2\sqrt{2C}\|\lambda^*\|}{s(k-1)(k+\alpha-3)}
\end{eqnarray*}
for all $k>1$.
\end{proof}

\begin{remark}
It was shown in \cite[Theorem 2.9]{Xu2017} that the algorithm \eqref{al0} with adaptive parameters enjoys $\mathcal{O}(1/k^2)$ rate.  Xu \cite{{Xu2017}} did not  discuss whether the convergence rate of algorithm \eqref{al0} is preserved  when the subproblem is solved inexactly. By Theorem \ref{th:cov_al3}, the  $\mathcal{O}(1/k^2)$  convergence rate of Algorithm \ref{al:al2} is preserved  even if  the subproblem  is solved inexactly, provided the errors are sufficiently small. The numerical experiments in section 3 also show the effectiveness of the inexact algorithm.
\end{remark}

\section{Numerical experiments}
In this section, we present  numerical experiments to illustrate the efficiency of the proposed methods. All codes are run on a PC (with 2.3 GHz Quad-Core Intel Core i5 and and 8 GB memory) under MATLAB Version 9.4.0.813654 (R2018a).

\subsection{The quadratic programming problem} In this subsection, we test the  algorithms on the nonnegative linearly constrained quadratic programming problem  (NLCQP):
\begin{equation*}
		\min \ \frac{1}{2}x^TQx+q^Tx, \quad s.t. \ Ax =  b, x\geq 0,
\end{equation*}
where $ q \in\mathbb{R}^n$, $Q\in\mathbb{R}^{n\times n}$ is a positive semidefinite  matrix,  $A\in\mathbb{R}^{m\times n}$,  $ b \in\mathbb{R}^m$.
 Here, we compare Algorithm \ref{al:al2} (Al2)  with the accelerated linearized augmented Lagrangian method (AALM \cite[Algorithm 1]{Xu2017}), which enjoys $\mathcal{O}(1/k^2)$ convergence rate with  adaptive parameters.

\begin{figure}[htbp]
\centering
\subfigure[$\emph{subtol}=10^{-6}$]{
\begin{minipage}[t]{0.32\linewidth}
\centering
\includegraphics[width=2in]{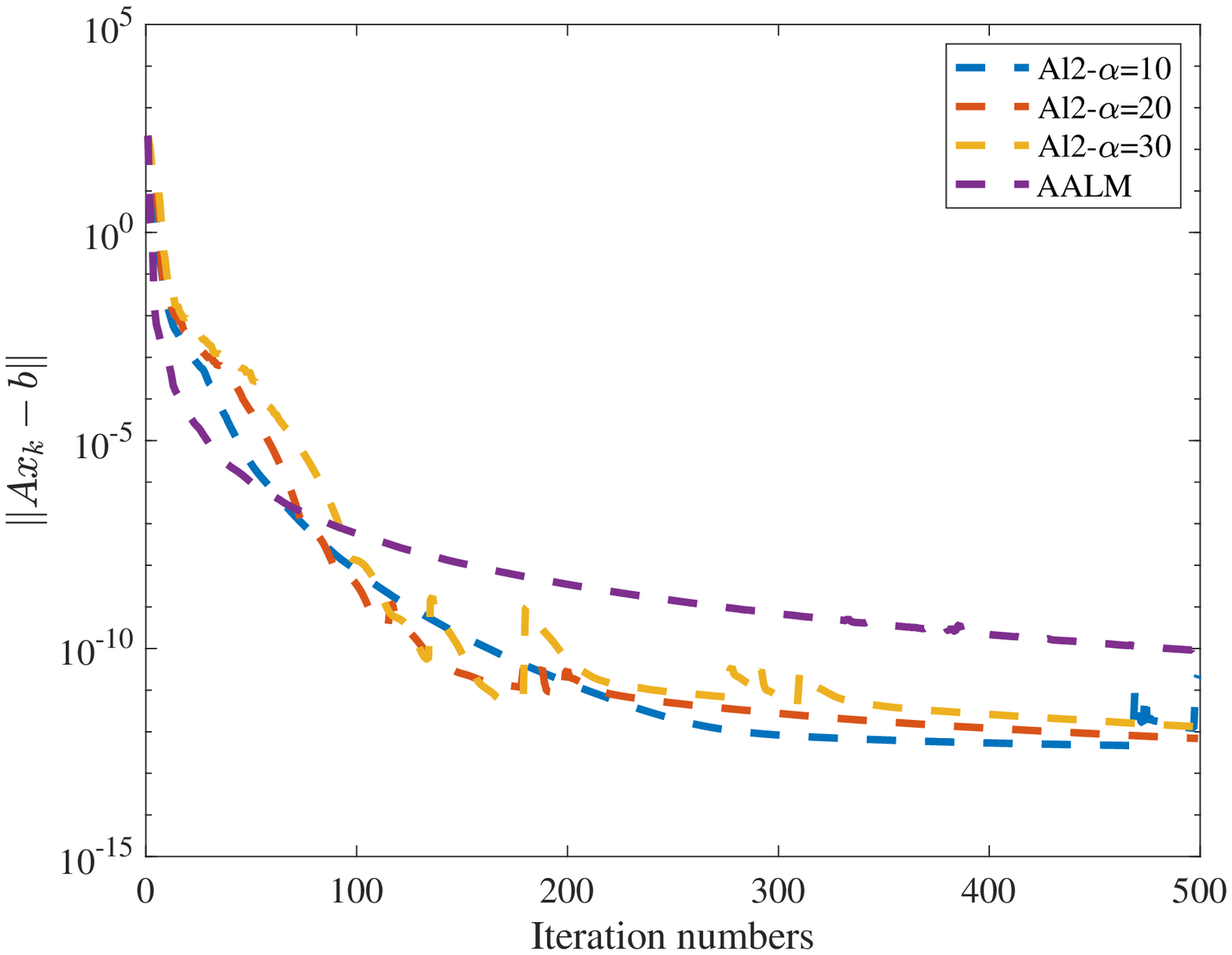}
%\caption{fig1}
\centering
\includegraphics[width=2in]{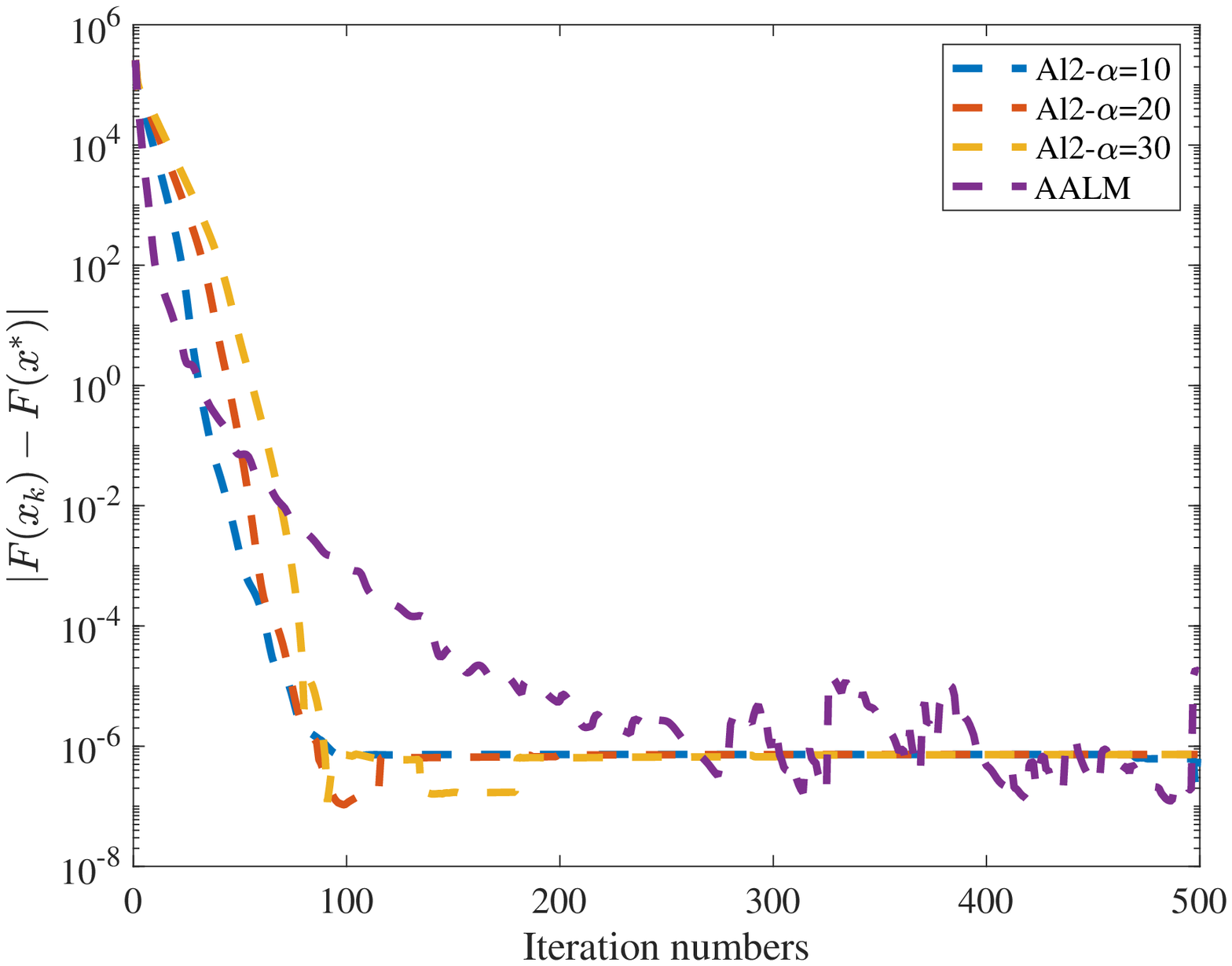}
%\caption{fig2}
\end{minipage}%
}%
\subfigure[$\emph{subtol}=10^{-8}$]{
\begin{minipage}[t]{0.32\linewidth}
\centering
\includegraphics[width=2in]{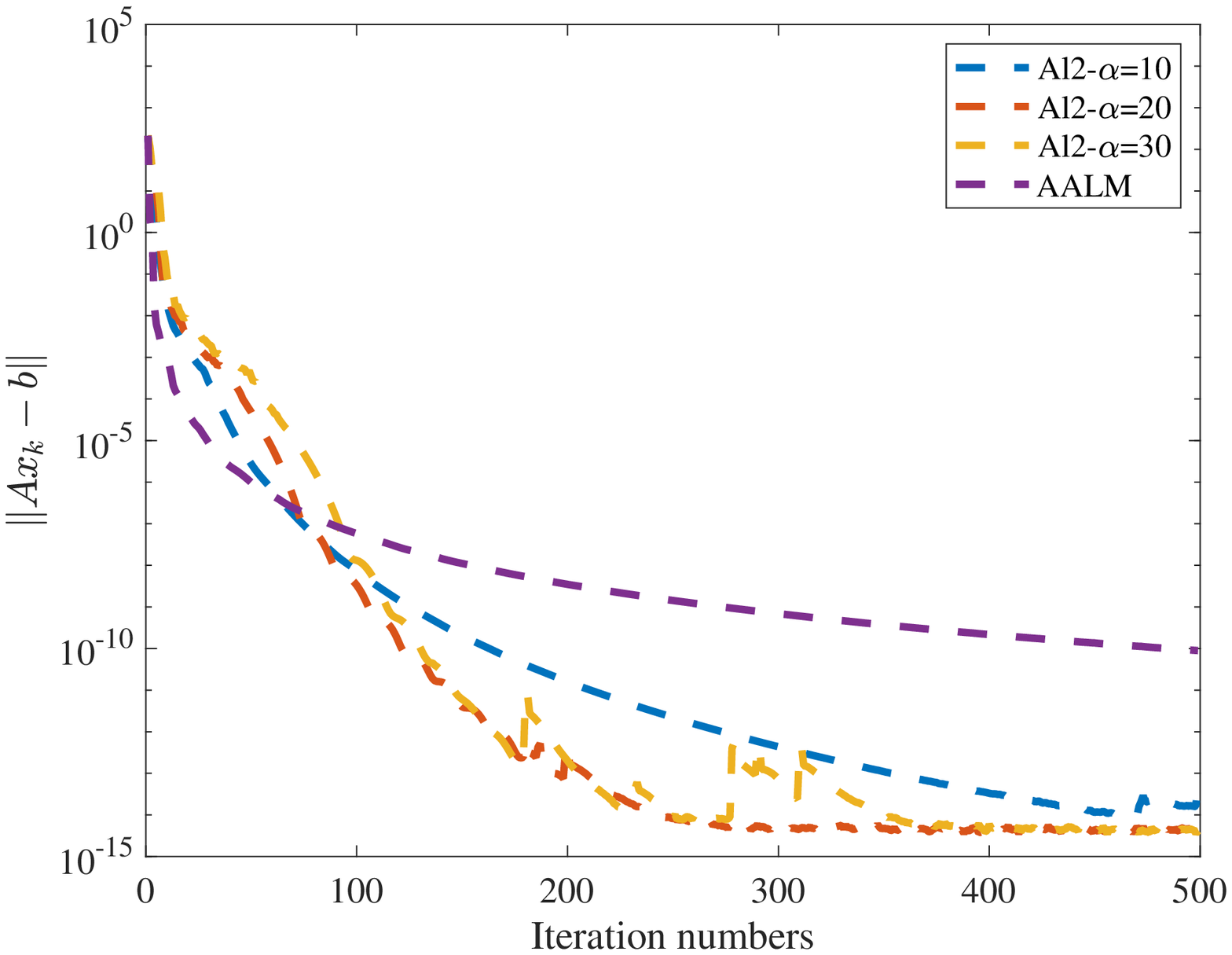}
%\caption{fig1}
\centering
\includegraphics[width=2in]{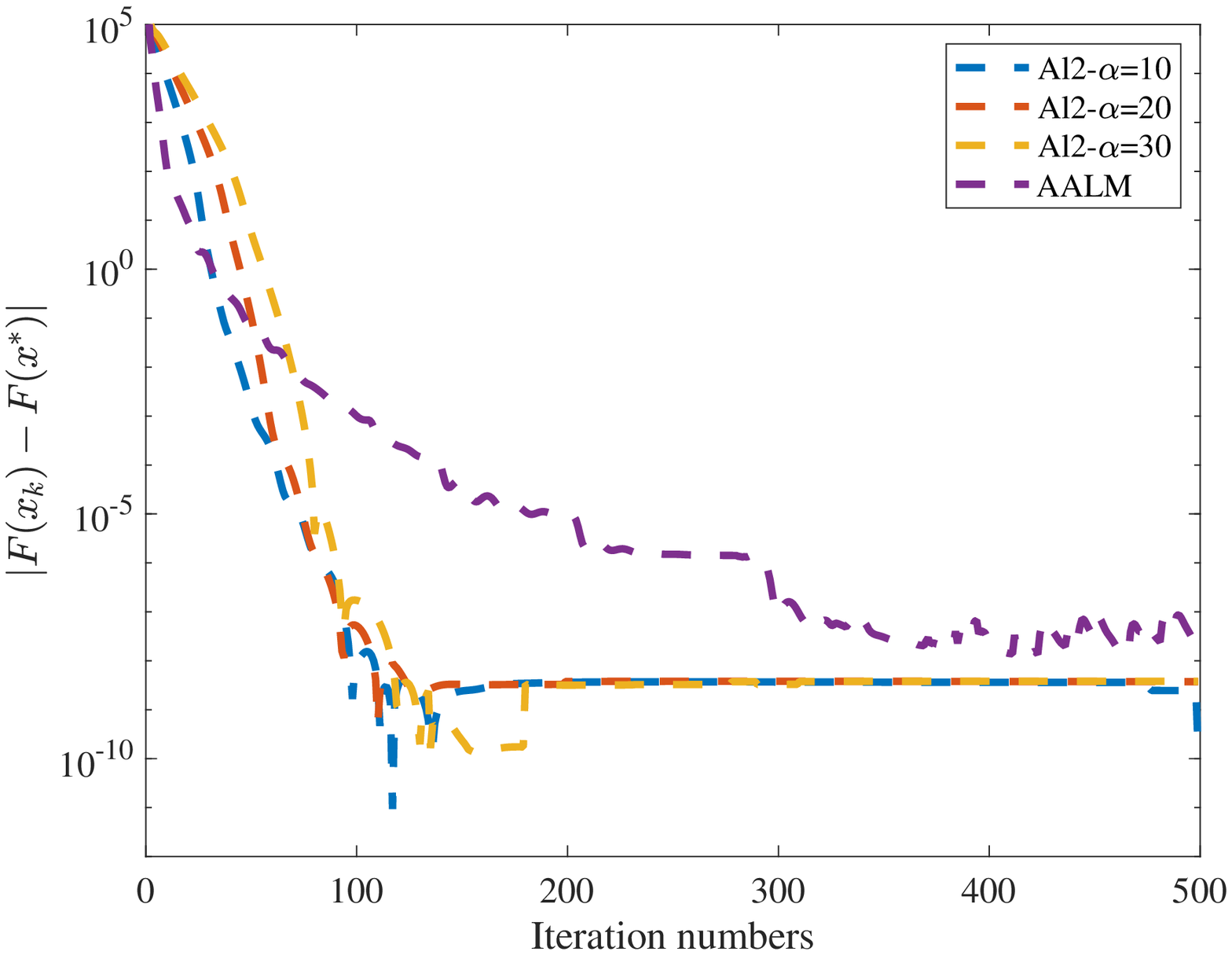}
%\caption{fig2}
\end{minipage}%
}%
\subfigure[$\emph{subtol}=10^{-10}$]{
\begin{minipage}[t]{0.32\linewidth}
\centering
\includegraphics[width=2in]{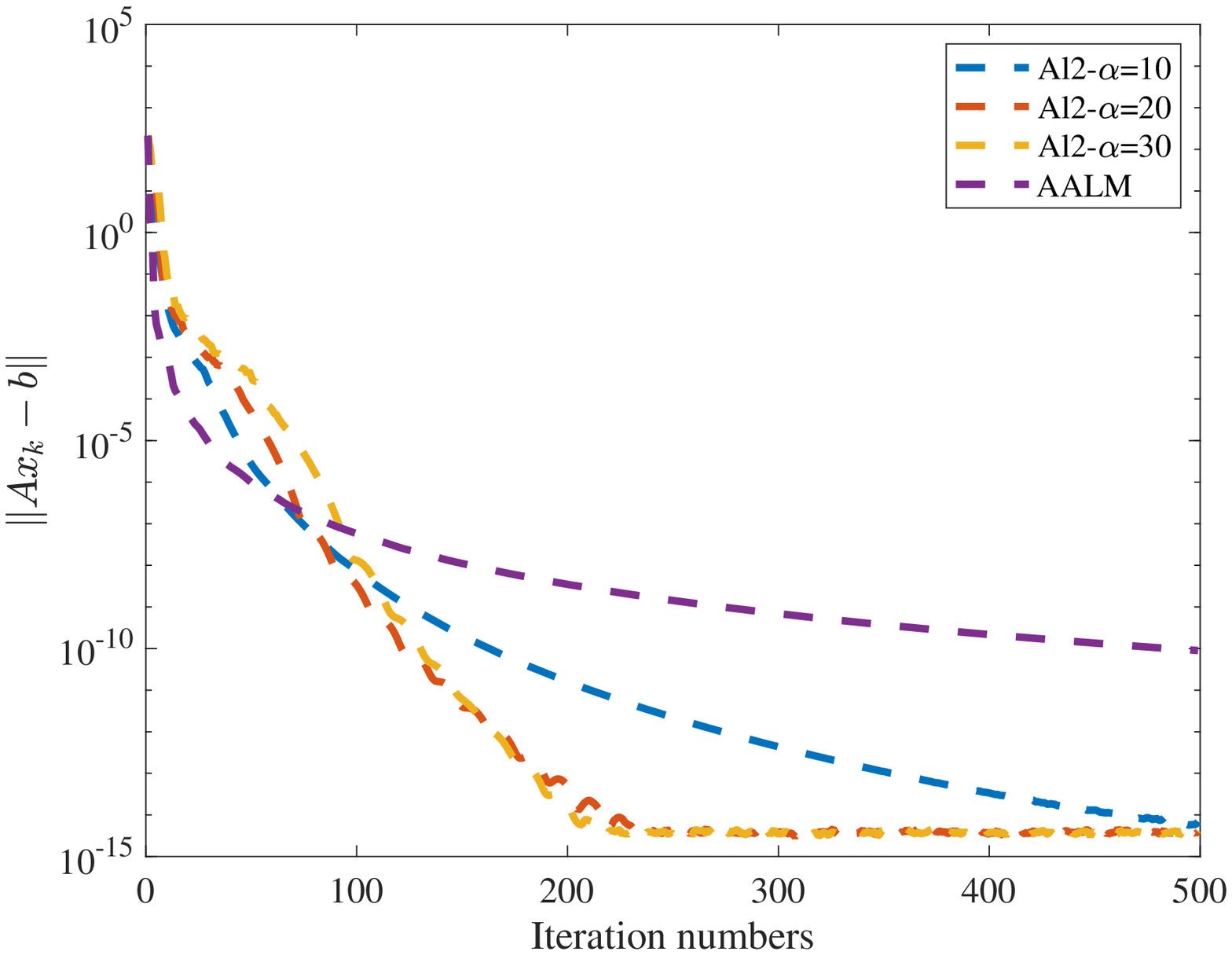}
%\caption{fig1}
\centering
\includegraphics[width=2in]{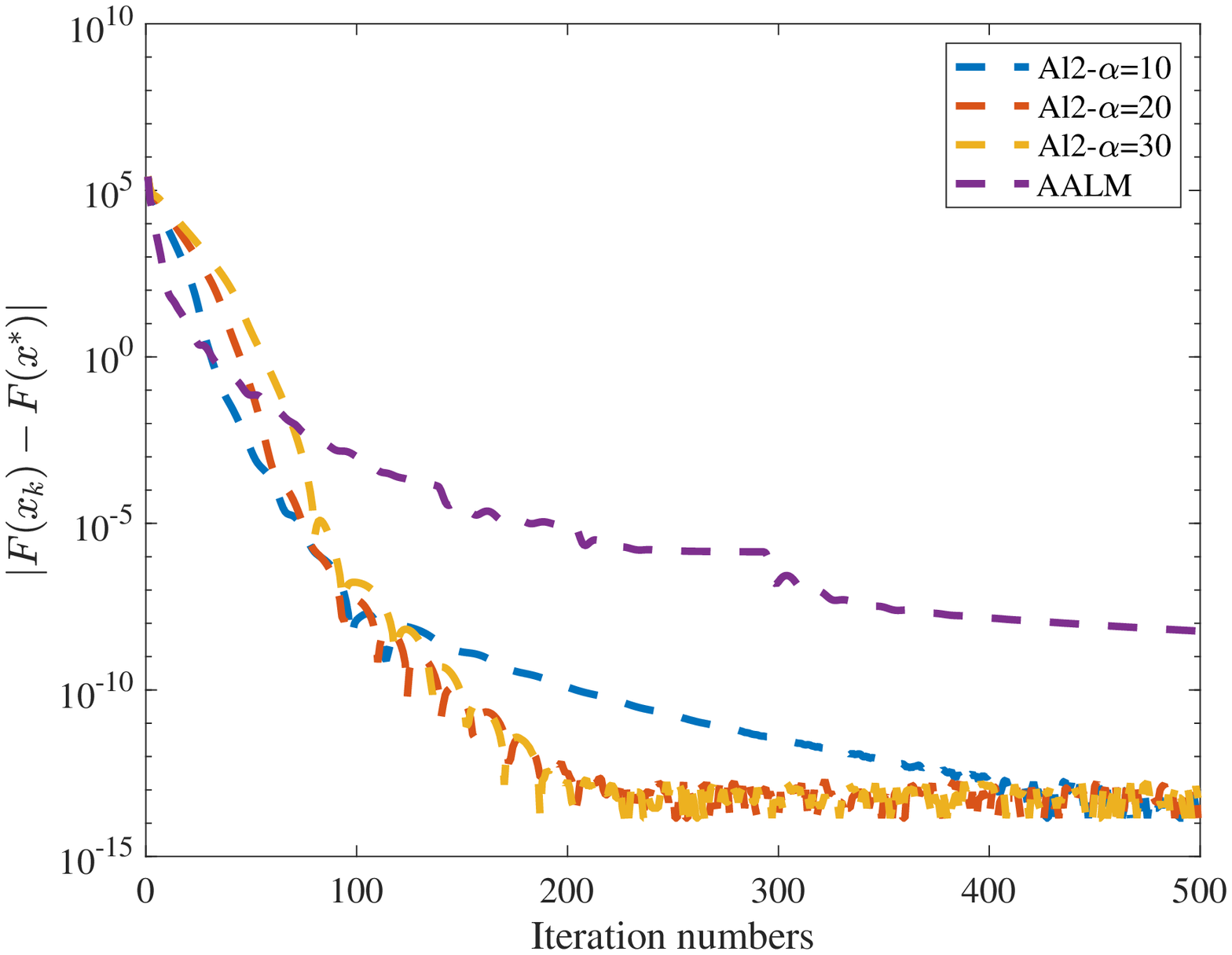}
%\caption{fig2}
\end{minipage}%
}%
\caption{Error of objective function and constraint of Al2 and AALM with different \emph{subtol} }\label{fig1}

\end{figure}

Set $m = 100$ and $n = 500$. Let $q$ be generated by standard Gaussian distribution, $b$ be generated by uniform distribution, $A=[B, Id]$ with  $B\in\mathbb{R}^{m\times(n-m)}$  generated by standard Gaussian distribution, $Q = 2H^TH$ with $H\in\mathbb{R}^{n\times n}$  generated  by standard Gaussian distribution. Then $Q$ may not be positive definite. The optimal value $F(x^*)$ is obtained by Matlab function $quadprog$ with tolerance $10^{-15}$.
 In this case, $F(x)=f(x)+g(x)$ with $f(x) = \mathcal{I}_{y\geq 0}(x)$,\ $g(x)= \frac{1}{2}x^TQx+q^Tx$, where $\mathcal{I}_{y\geq 0}$ is the indicator function of the set $\{y|y\geq 0\}$, i.e.,
\begin{equation*}
\mathcal{I}_{y\geq 0}(x)=
	\begin{cases}
		0,\qquad &x\geq 0,\\
		+\infty \qquad &otherwise.
	\end{cases}
\end{equation*}
Set the parameters of Algorithm \ref{al:al2} as: $s=\|Q\|$, $M_k=s*\|Q\|Id$. Set the parameters of AALM (\cite[Algorithm 1]{Xu2017}) with  adaptive parameters, in which $\alpha_k=\frac{2}{k+1}$, $\beta_k=\gamma_k=\|Q\|k$, $P_k=\frac{2\|Q\|}{k}Id$. Subproblems for both algorithms are solved by interior-point algorithms to a tolerance \textit{subtol}. Figure 1 describes  the distance of optimal value $|F(x_k)-F(x^*)|$ and violation of feasibility $\|Ax_k-b\|$ given Al2 with $\alpha = 10,\ 20,\ 30$ and AALM for the first 500 iterations. As shown in Figure 1, Algorithm 2 performs better and more stable than AALM under different \textit{subtol}.

\subsection{The basis pursuit problem}
Consider the following basis pursuit problem:
\begin{equation*}
		\min_x  \quad \|x\|_1, \quad s.t.  \  Ax = b,
\end{equation*}
where $A\in \mathbb{R}^{m\times n}$, $b\in\mathbb{R}^m$ and $m\leq n$. Let  $A$ be generated by  standard Gaussian distribution. The number of nonzero elements of the original solution  $x^*$ is fixed at $0.1*n$, and the  nonzero elements are selected randomly in $[-2,2]$. Set $b=Ax^*$. We compare Algorithm \ref{al:al1} with the inexact  augmented Lagrangian method (IAL \cite[Algorithm 1]{Liu2019}). Here, subproblems for both algorithms are solved by fast iterative shrinkage-thresholding algorithm (FISTA \cite{BeckT2009}, \cite[Algorithm 2]{Liu2019}), and the stopping condition of the FISTA is when
\[ \frac{\|x_k-x_{k-1}\|^2}{\max\{\|x_{k-1}\|,1\}}\leq subtol \]
is satisfied or the number of iterations exceeds $100$, where accuracy $subtol=1e-4,\ 1e-6,\ 1e-8$. In each test, we calculate the residual error $\|Ax_k-b\|$ ($Res$) and the relative error of the solution $\frac{\|x_k-x^*\|}{\|{x}^*\|}$ ($Rel$) with the  stopping condition $Res+Rel\leq 1e-8$. Set the parameters of Algorithm \ref{al:al1} as $\alpha=n$, $s=100$, $M_k=0$, and the parameter of IAL as $\beta=1$. Let  $Init$ and $Time$ denote the number of iterations,  and the CPU time in seconds, respectively. Under different tolerance \textit{subtol} of subproblem, Table \ref{tab1}-Table \ref{tab3} report the results for the basis pursuit problem with different dimensions. We  observe that when the subproblem is solved with different accuracy, Algorithm \ref{al:al1}  is  faster than IAL in terms of the number of iterations and the  cpu time.

\begin{table}[!h]
		\centering
		\caption{Numerical results of Algorithm \ref{al:al1} and IAL with $subtol =1e-4 $}\label{tab1}
		\setlength{\tabcolsep}{3mm}
		\begin{tabular}{cccccccccc}
			\hline
			\multirow{2}{*}{$ID$}&\multicolumn{4}{c} {Algorithm \ref{al:al1}} & &\multicolumn{4}{c}{IAL}\\
			\cline{2-5}\cline{7-10}
			&$Res$&$Rel$&$Init$&$Time$ & &$Res$&$Rel$&$Init$&$Time$\\
			\hline
			$m=60,n=100$& 8.0e-9&3.5e-10&158&0.07&& 7.4e-9&3.2e-10&186& 0.09 \\
			\cline{2-5}\cline{7-10}
			$m=200,n=300$& 7.4e-9&1.1e-10&231&0.78&& 9.3e-9&1.4e-10&281& 0.92 \\
			\cline{2-5}\cline{7-10}
			$m=300,n=500$& 8.3e-9&7.8e-11&278&2.26&& 9.1e-9&8.5e-11&322& 2.63 \\
			\cline{2-5}\cline{7-10}
			$m=600,n=1000$& 7.5e-9&3.3e-11&300&10.73&& 7.8e-9&3.5e-11&374& 13.78\\
			\cline{2-5}\cline{7-10}
			$m=1000,n=1500$& 8.7e-9&2.6e-11& 284&35.12&& 7.6e-9&2.3e-11&327& 46.28\\
			\hline
		\end{tabular}
	\end{table}
	
\begin{table}[!h]
		\centering
		\caption{Numerical results of Algorithm \ref{al:al1} and IAL with $subtol =1e-6$}\label{tab2}
		\setlength{\tabcolsep}{3mm}
		\begin{tabular}{cccccccccc}
			\hline
			\multirow{2}{*}{$ID$}&\multicolumn{4}{c} {Algorithm \ref{al:al1}} & &\multicolumn{4}{c}{IAL}\\
			\cline{2-5}\cline{7-10}
			&$Res$&$Rel$&$Init$&$Time$ & &$Res$&$Rel$&$Init$&$Time$\\
			\hline
			$m=60,n=100$& 9.0e-9&3.9e-10&86&0.05&& 6.2e-9&2.6e-10&130& 0.08 \\
			\cline{2-5}\cline{7-10}
			$m=200,n=300$& 6.0e-9&8.6e-10&140&0.48&& 9.2e-9&1.4e-10&174& 0.58 \\
			\cline{2-5}\cline{7-10}
			$m=300,n=500$& 8.8e-9&8.0e-11&185&1.61&& 9.5e-9&8.9e-11&215& 1.78 \\
			\cline{2-5}\cline{7-10}
			$m=600,n=1000$& 7.8e-9&3.6e-11&232&8.34&& 8.4e-9&3.7e-11&262& 9.68\\
			\cline{2-5}\cline{7-10}
			$m=1000,n=1500$& 9.6e-9&2.7e-11& 193&24.25&& 8.6e-9&2.6e-11&277& 34.67\\
			\hline
		\end{tabular}
\end{table}

\begin{table}[!h]
		\centering
		\caption{Numerical results of Algorithm \ref{al:al1} and IAL with $subtol =1e-8$}\label{tab3}
		\setlength{\tabcolsep}{3mm}
		\begin{tabular}{cccccccccc}
			\hline
			\multirow{2}{*}{$ID$}&\multicolumn{4}{c} {Algorithm \ref{al:al1}} & &\multicolumn{4}{c}{IAL}\\
			\cline{2-5}\cline{7-10}
			&$Res$&$Rel$&$Init$&$Time$ & &$Res$&$Rel$&$Init$&$Time$\\
			\hline
			$m=60,n=100$& 8.5e-9&3.7e-10&32&0.03&& 8.7e-9&3.0e-10&37& 0.03 \\
			\cline{2-5}\cline{7-10}
			$m=200,n=300$& 8.3e-9&1.2e-10&85&0.28&& 8.4e-9&1.2e-10&100& 0.34 \\
			\cline{2-5}\cline{7-10}
			$m=300,n=500$& 9.8e-9&9.0e-11&95&0.84&& 9.5e-9&8.5e-11&121& 1.00 \\
			\cline{2-5}\cline{7-10}
			$m=600,n=1000$& 9.0e-9&4.2e-11&108&4.32&& 8.1e-9&3.3e-11&136& 5.47\\
			\cline{2-5}\cline{7-10}
			$m=1000,n=1500$& 9.8e-9&2.8e-11& 108&13.63&& 8.6e-9&2.4e-11&144& 18.67\\
			\hline
		\end{tabular}
\end{table}
	
\subsection{The linearly constrained $\ell_1-\ell_2$ minimization problem}
Consider the following problem:
\begin{equation*}
	\min_x\ \|x\|_1 +\frac{\beta}{2}\|x\|^2_2 \quad s.t.\ Ax =  b,
\end{equation*}
where $A\in \mathbb{R}^{m\times n}$ and $b\in\mathbb{R}^m$. Let $m = 1500, n=3000$, and $A$ be generated by  standard Gaussian distribution. Suppose that the original solution (signal) ${x}^*\in\mathbb{R}^n$  has only 150 non-zero elements which are generated by the Gaussian distribution $\mathcal{N}(0,4)$ in the interval $[-2,2]$ and that the noise $\omega$ is selected randomly with $\|\omega\|=10^{-4}$,
\[ b= A{x}^*+\omega.\]
Set parameters  for Algorithm \ref{al:al2} (Al2) with $\alpha=20$, $s=1$, $M_k = s\beta$, and the parameters of IAALM (\cite[Algorithm 1]{KangJ2015}) with  $\gamma=1$. Subproblems  are solved by FISTA and the stopping condition is when
\[ \frac{\|x_k-x_{k-1}\|^2}{\max\{\|x_{k-1}\|,1\}}\leq subtol \]
is satisfied or the number of iterations exceeds $100$, where accuracy $subtol= 1e-6,\ 1e-8$. We terminate all the methods when  $\|Ax_k-b\|\leq 5*10^{-4}$.
In each test, we calculate the residual error $res = \|Ax-b\|$,  the relative error $rel = \frac{\|x-x^*\|}{\|{x}^*\|}$ and the signal-to-noise ratio
\[SNR=\log_{10} \frac{\|x^*-\text{mean}(x^*)\|^2}{\|x-{x}^*\|^2},\]
where $x$ is the recovery signal.

\begin{table}[!h]
		\centering
		\caption{Numerical results of Algorithm \ref{al:al2} and IAALM with $subtol = 1e-8$}\label{tab4}
		\setlength{\tabcolsep}{4mm}
		\begin{tabular}{cccccccc}
			\hline
			&$\beta$ & 0.01 &0.05 & 0.1 & 0.5 & 1 & 1.5\\
			\hline
			\multirow{2}{*}{$Init$}& Al2&10&10&10&13&15&42\\
			& IAALM&35&37&34&33&35&100+\\
			\hline
			\multirow{2}{*}{$Time$}& Al2&13.15&11.80&10.47&12.38&14.99&46.75\\
			& IAALM&38.12&36.77&35.34&36.50&34.54&106.45\\
			\hline
			\multirow{2}{*}{$Res$}& Al2&4.18e-4&3.37e-4&3.18e-4&2.47e-4&4.74e-4&4.34e-4\\
			& IAALM&4.91e-4&4.82e-4&4.75e-4&4.22e-4&4.21e-4&6.50e-3\\
			\hline
			\multirow{2}{*}{$Rel$}& Al2&1.31e-6&8.01e-7&6.35e-7&4.18e-7&9.09e-7&7.58e-2\\
			& IAALM&9.05e-7&8.47e-7&9.19e-7&8.15e-7&9.62e-7&7.34e-2\\
			\hline
			\multirow{2}{*}{$SNR$}& Al2&1.17e+2&1.22e+2&1.24e+2&1.28e+2&1.21e+2&2.24e+1\\
			& IAALM&1.20e+2&1.21e+2&1.21e+2&1.22e+2&1.20e+2&2.27e+1\\
			\hline
		\end{tabular}	
\end{table}

\begin{table}[!h]
		\centering
		\caption{Numerical results of Algorithm \ref{al:al2} with $subtol = 1e-6$ }\label{tab5}
		\setlength{\tabcolsep}{4mm}
		\begin{tabular}{cccccccc}
			\hline
			$\beta$ & 0.01 &0.05 & 0.1 & 0.5 & 1 & 1.5\\
			\hline
			{$Init$}&25&22&17&13&18&100+\\
			\hline
			{$Time$}&8.29&8.11&7.52&8.65&12.74&36.48\\
			\hline
			{$Res$}&4.62e-4&4.89e-4&4.93e-4&4.36e-4&4.36e-4&1.50e-3\\
			\hline
			{$Rel$}&1.52e-6&1.68e-6&1.30e-6&7.17e-7&9.12e-7&8.10e-2\\
			\hline
			{$SNR$}&1.16e+2&1.15e+2&1.18e+2&1.22e+2&1.21e+2&2.18e+1\\
			\hline
		\end{tabular}	
\end{table}
\begin{figure}[h]
\centering
\subfigbottomskip=1pt
\subfigcapskip=-10pt
\subfigure[Recovery with $\beta = 0.5$]{\includegraphics[width=0.45\linewidth,height=4cm]{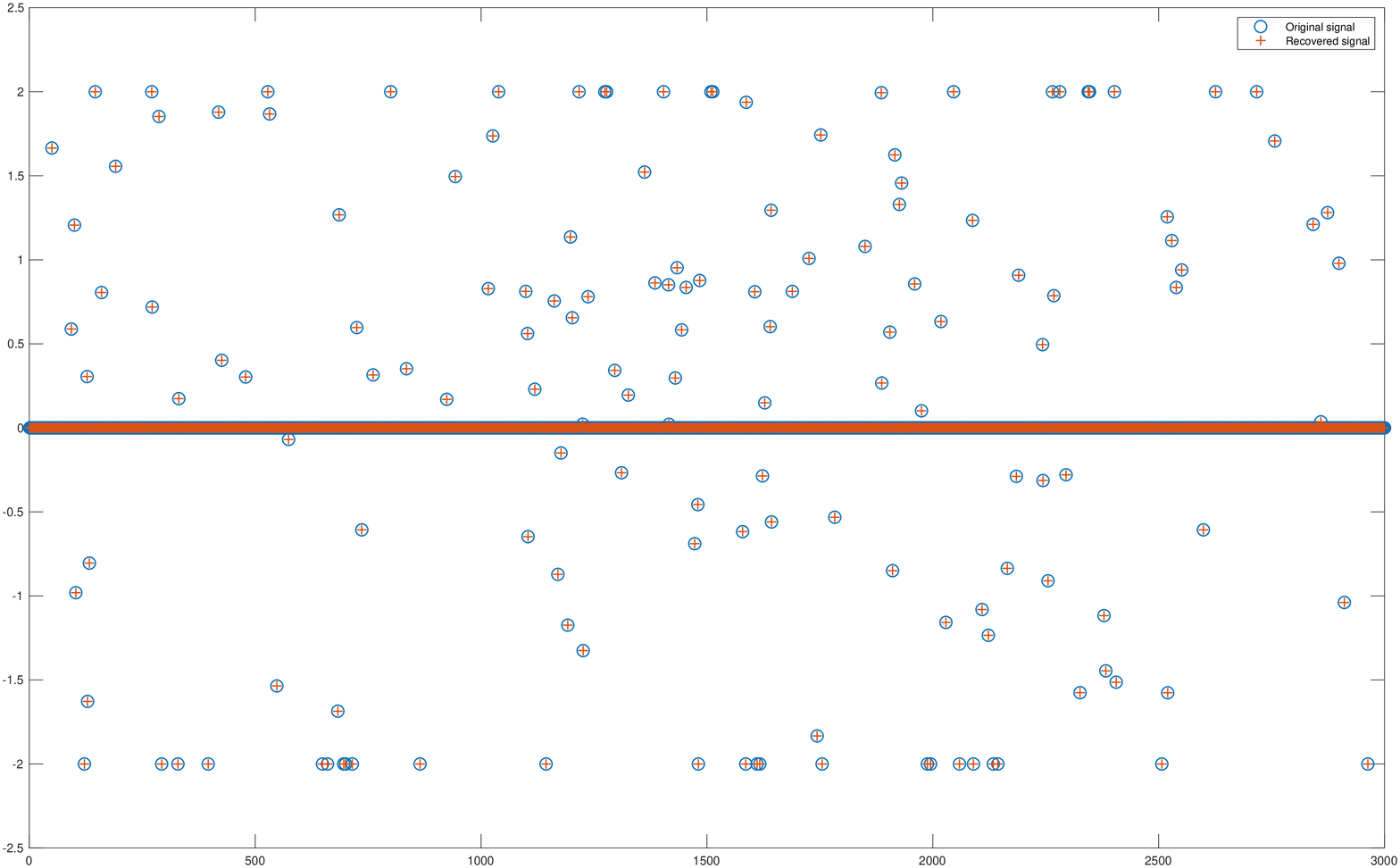}}\quad
\subfigure[Recovery with $\beta = 1.5$]{\includegraphics[width=0.45\linewidth,height=4cm]{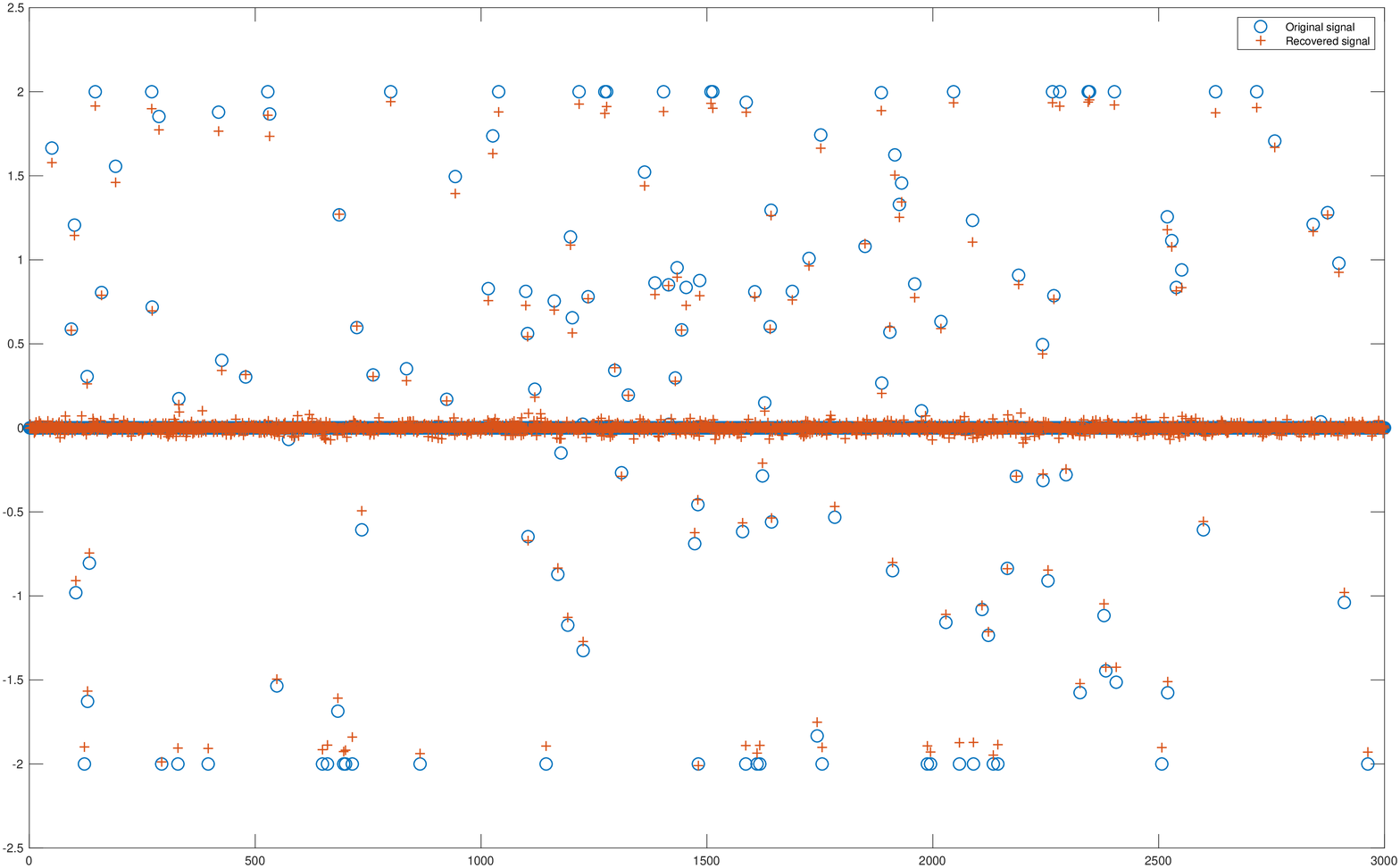}}
\caption{ Original sparse signal and the final estimated solution of Algorithm \ref{al:al2} with $subtol=1e-8$.}\label{fig_2}
\end{figure}

In Table \ref{tab4},  we present the  numerical results of Algorithm 2 and IAALM for various $\beta$. When $subtol=1e-6$, IAALM doesn't work well, we list  the  numerical results of Algorithm 2 in Table \ref{tab5}. Based on the $Rel$ and $SNR$, it is seen that the sparse original signal is well restored when $\beta\leq 1$. This is also shown in Figure \ref{fig_2}.

\section{Conclusion}
In this paper, we propose two inertial accelerated primal-dual methods for solving linear equality constrained convex optimization problems. Assuming merely convexity, we show the inertial primal-dual methods own $\mathcal{O}(1/k^2)$ convergence rates even if the subproblem is solved inexactly. The numerical results demonstrate the validity  and superior performance of our methods over  some existing methods.
%\begin{acknowledgements}
%\end{acknowledgements}

\end{document}